\documentclass[12pt]{amsart}
\usepackage[top=2cm, bottom=2cm, left=2cm, right=2cm]{geometry}

\usepackage{hyperref}
\usepackage{graphicx, pb-diagram}

\author[A. D. Smith]{Abraham D. Smith}
\address{McGill University, Department of Mathematics and Statistics\\
         Montreal, Quebec  H3A 2K6\\
         Canada}
\email{\href{mailto:adsmith@msri.org}{adsmith@msri.org}}
\urladdr{\href{http://www.math.mcgill.ca/adsmith}%
         {http://www.math.mcgill.ca/adsmith}}

\thanks{This work is supported by an NSF All-Institutes
Postdoctoral Fellowship administered by the Mathematical Sciences Research
Institute through its core grant DMS-0441170. The author is hosted by the
Department of Mathematics and Statistics at McGill University.}

\title[GL(2) PDEs]
{Integrable GL(2) Geometry and Hydrodynamic Partial Differential Equations}
\date{\today}

\newtheorem{thm}{Theorem}[section]
\newtheorem{lemma}[thm]{Lemma}
\newtheorem{cor}[thm]{Corollary}
\newtheorem{defn}{Definition}[section]
\newcommand{\lhk}{\mathbin{\hbox{\vrule height1.4pt width4pt depth-1pt
\vrule height4pt width0.4pt depth-1pt}}} 

\newcommand{\pair}[1]{\ensuremath{\left\langle #1 \right\rangle}}

\begin{document}

\begin{abstract}
This article is a local analysis of integrable $GL(2)$-structures of degree 4.
A $GL(2)$-structure of degree $n$ corresponds to a distribution of rational
normal cones over a manifold of dimension $n+1$.  Integrability corresponds to
the existence of many submanifolds that are spanned by lines in the cones.
These $GL(2)$-structures are important because they naturally arise from a
certain family of second-order hyperbolic PDEs in three variables that are
integrable via hydrodynamic reduction.  Familiar examples include the wave
equation, the first flow of the dKP hierarchy, and the Boyer--Finley equation.

The main results are a structure theorem for integrable $GL(2)$-structures, a
classification for connected integrable $GL(2)$-structures, and an equivalence
between local integrable $GL(2)$-structures and Hessian hydrodynamic hyperbolic
PDEs in three variables.

This yields natural geometric characterizations of the wave equation, the
first flow of the dKP hierarchy, and several others.  It also provides an
intrinsic, coordinate-free infrastructure to describe a large class of hydrodynamic
integrable systems in three variables. 
\end{abstract}

\maketitle

\tableofcontents

\section*{Introduction}
A fundamental problem in analysis is to understand why some differential
equations (particularly hyperbolic equations arising from wave-like equations
or from differential geometry) are \emph{integrable}.  This problem is
compounded by the many competing notions of integrability \cite{Zakharov1991}.
Informally, a PDE is called integrable if it can be solved thanks to the
existence of an infinite hierarchy of conservation laws or of a family of
invariant foliations along characteristics.  Integrable PDEs of all types are
well-studied in two (${1+1}$) independent variables; however, examples are
increasingly rare and poorly understood in higher dimensions.

A particularly interesting class of integrable PDEs are those that can be
integrated by decomposing the equation to a set of coupled first-order
equations that represent traveling waves \cite{Tsarev1990}. This technique is
called ``integration by hydrodynamic reduction,'' and it has been extensively
studied for various special classes of second-order PDEs in three or more
independent variables \cite{Burovskiy2008, Ferapontov2009, Ferapontov2003, Ferapontov2004,
Ferapontov2002, Tsarev1993, Tsarev2000}

Consider a second-order hyperbolic equation of the form $F(\xi^i, u, u_i, u_{ij})=0$ on
scalar functions $u(\xi^1, \xi^2, \xi^3)$.   In the method of hydrodynamic
reduction, one hopes to integrate the PDE
${F=0}$ by stipulating that the solution function $u$ (and its derivatives) may
be written as $u(R^1, \ldots, R^k)$ for an \emph{a priori} unknown number $k$
of functions $R^1(\xi), \ldots, R^k(\xi)$ whose derivatives admit the
``commuting'' relations  
\begin{equation}
\label{eqn:hydro}
\frac{\partial}{\partial \xi^2}R^i = \rho_{2}^i(R)\frac{\partial}{\partial
\xi^1}R^i, \qquad
\frac{\partial}{\partial \xi^3}R^i = \rho_{3}^i(R)\frac{\partial}{\partial
\xi^1}R^i
\end{equation}
which also imply
\begin{equation}
\frac1{\rho^i_{2} - \rho^j_{2}}\frac{\partial \rho^i_{2}}{\partial R^j} =
\frac1{\rho^i_{3} - \rho^j_{3}}\frac{\partial \rho^i_{3}}{\partial R^j},
\quad \forall i \neq j.
\label{eqn:hydro2}
\end{equation}
Equations~(\ref{eqn:hydro}) and (\ref{eqn:hydro2}) can be solved by other
methods \cite{Tsarev1990}.  The reduction of ${F=0}$ to
Equation~(\ref{eqn:hydro2}) is called a $k$-parameter hydrodynamic reduction.
The equation ${F=0}$ is called integrable via $k$-parameter hydrodynamic
reduction if it admits an infinite family of $k$-parameter hydrodynamic
reductions and this family is itself parametrized by $k$ functions of one
variable.  The functions $R^i(\xi^1,\xi^2,\xi^3)$ are called Riemann
invariants, and their level sets define the foliations along characteristics
in the informal definition above. See the works cited above, particularly
\cite{Ferapontov2009}, \cite{Ferapontov2003} and \cite{Ferapontov2002}, for
examples and detailed exposition of this technique.  

Many examples of these so-called hydrodynamic equations are known, such as the
wave equation $u_{22}=u_{13}$, the dispersionless Kadomtsev--Petviashvili (dKP)
equation $u_{22}=(u_1- uu_3)_3$, the first flow of the dKP hierarchy
$u_{22}=u_{13}-\frac12(u_{33})^2$, the Boyer--Finley equation $u_{22}+u_{33} =
e^{u_{11}}$, and some well-known integrable hierarchies. Tests for
hydrodynamic integrability also exist, but so far there is no intrinsic,
coordinate-free theory that classifies and generates all of these equations.

For the case of second-order hydrodynamic equations in three variables that
involve only the Hessian of the solution function, recent work by Ferapontov,
\emph{et al.}, reveals a relationship between the integrability of the PDEs, the
symplectic contact symmetries of the equations, and the equations' underlying degree-4
$GL(2)$-structures \cite{Ferapontov2009}.  Integrable $GL(2)$-structures are not
well-known in the literature; they arise in the search for exotic affine
holonomies in differential geometry \cite{Bryant1991a} and a similar conformal
geometry appears useful in the equivalence problem for ODEs
\cite{Nurowski2007}, but the application to PDEs is very recent
\cite{Ferapontov2009, Smith2009}. 

The purpose of this article is to fully analyze the local geometry of
integrable $GL(2)$-structures of degree 4, with an eye towards understanding
PDEs that are integrable via hydrodynamic reduction.  The main result is a
natural classification of local integrable $GL(2)$-structures and the
associated PDEs.  This is achieved using geometry \emph{\`a la Cartan}, namely
exterior differential systems, moving frames, and the method
of equivalence.

Here is a summary of the contents of this article, including abbreviated and
\emph{non-technical} versions of the main theorems:

Section~\ref{background} defines $GL(2)$-structures and their 2-integrability
and 3-integrability.   The property of $k$-integrability involves the
existence of ``many'' foliations by certain submanifolds of dimension $k$.
Section~\ref{background} also introduces the motivating observation by
Ferapontov \emph{et al.} that hyperbolic PDEs in three variables that are
integrable by means of hydrodynamic reductions lead to $GL(2)$-structures that
are both 2-integrable and 3-integrable (abbreviated 2,3-integrable).

Section~\ref{connection} contains a description of the necessary $SL(2)$
representation theory and an explicit solution to the first-order equivalence problem for
$GL(2)$-structures.  Corollary~\ref{coframe} defines a global preferred
connection for $GL(2)$-structures that has essential torsion in
the direct sum of $\mathbb{R}^3$, $\mathbb{R}^7$, $\mathbb{R}^9$, and
$\mathbb{R}^{11}$, each of which is an irreducible $SL(2)$ module.

In Section~\ref{2int}, the exterior differential system describing
2-integrability for a $GL(2)$-structure is analyzed.  
The most important object
in this article, the irreducible $\mathbb{R}^9$-valued torsion, is first emphasized in
Theorem~\ref{int2-n4}.
\newtheorem*{thm2str}{Theorem~\ref{int2-n4}}
\begin{thm2str}
A $GL(2)$-structure is 2-integrable if and only if its torsion takes values only
in $\mathbb{R}^9$.  
\end{thm2str}

Section~\ref{3int} provides Theorem~\ref{int3-n4-str}, a complete local
description of $GL(2)$-structures that are 2,3-integrable. 
The immediate consequence of this theorem is that the value of the torsion at
a single point completely determines the local structure of a 2,3-integrable
$GL(2)$-structure.
\newtheorem*{thm23str}{Theorem~\ref{int3-n4-str}}
\begin{thm23str}
A $GL(2)$-structure is 2,3-integrable if and only if its structure equations are
of the form
\begin{equation}
\begin{split}
\mathrm{d}\omega &= -\theta\wedge\omega + T(\omega\wedge\omega) \\
\mathrm{d}\theta &= -\theta\wedge\theta + T^2(\omega\wedge\omega) \\
\mathrm{d}T &= J(T)(\omega,\theta)
\end{split}
\end{equation}
for an essential torsion $T$ in $\mathbb{R}^9$ and a $9 \times 9$ matrix
$J(T)$ depending on $T$, as in Appendix~\ref{appxJ}.
\end{thm23str}

Section~\ref{moduli} presents a classification of connected 2,3-integrable
$GL(2)$-structures.  This classification is provided by a natural stratification
of $\mathbb{R}^9$ into equivalence classes that are defined by the matrix
$J(T)$.  These equivalence classes are explicitly identified in
Theorem~\ref{classification}.
\newtheorem*{thmClass}{Theorem~\ref{classification}}
\begin{thmClass}
Every connected 2,3-integrable $GL(2)$-structure belongs to exactly one of 55
equivalence classes, which are given by a stratification of $\mathbb{R}^9$
into the root factorization types of real binary octic polynomials.
\end{thmClass}

In Section~\ref{PDEs}, the structure equations from Section~\ref{3int} are
applied to reproduce the integrable PDEs, yielding Theorem~\ref{embed}.
\newtheorem*{thmAllPDE}{Theorem~\ref{embed}}
\begin{thmAllPDE}
Every local 2,3-integrable $GL(2)$-structure arises from a
Hessian hydrodynamic hyperbolic PDE.
\end{thmAllPDE}
Therefore, the classification in Theorem~\ref{classification} is also a
classification of Hessian hydrodynamic hyperbolic PDEs.  Some previously known
examples of such PDEs are identified in the new classification, and several
new integrable PDEs are constructed.  Also, a relationship between Hessian
hydrodynamic hyperbolic PDEs in three variables and hyperbolic PDEs in two
variables (Monge--Amp\`ere, Goursat, or generic type) is observed.

This work would have been impossible to complete without the uncanny speed and
accuracy of computer algebra software, and there is no sense in retyping
dozens of pages of formulas and identities.  Thus, a reader looking for
explicit computational
details may be disappointed by this article (in which case I direct the
reader to \cite{Smith2009}).  However, most proofs rely only on basic concepts from
differential geometry, linear algebra, and representation theory, and a
reader comfortable with the standard methods of exterior differential systems
(and their various software implementations) can readily supply the
computational details.  A \textsc{Maple} file containing the structure
equations appearing in Theorem~\ref{int3-n4-str} and the Maurer--Cartan form
appearing in Theorem~\ref{embed} will remain available at \cite{GL2maplecode}
as long as possible.  This file should also be bundled with this article on
\href{http://arxiv.org/abs/0912.2789}{arXiv.org}.  The content of this file allows
rapid reproduction of all results following Theorem~\ref{int3-n4-str}.

Please note that the terms ``hyperbolic'' and ``integrable'' are heavily used
in slightly different contexts throughout this article.  ``Hyperbolic'' refers
to hyperplanes having maximal intersection with certain projective
varieties (for example, in Lemma~\ref{hyperbolics}), to PDEs in both two
and three independent variables with appropriate leading symbol (for example, the
wave equation), and to exterior differential systems with appropriate tableaux
(for example, in the proof of Theorem~\ref{bisecant}).  The term
``integrable'' refers both to PDEs that admit infinitely many exact solutions
and to $GL(2)$-structures that admit many secant submanifolds (submanifolds
which one might be tempted to also call ``hyperbolic'').  Of course, all of
these concepts are intimately related, so the standard overloading of these
definitions is reasonable.

I owe many thanks to Robert Bryant and Niky Kamran for their wisdom,
guidance, and support, to Dennis The and Eugene Ferapontov for their many helpful comments, 
and to Jeanne Clelland, whose ``Cartan'' package for
\textsc{Maple} made the computations bearable.

\section{Background}
\label{background}

Let $\mathcal{V}_n$ denote the vector space of degree $n$ homogeneous
polynomials in $x$ and $y$ with real coefficients.  Identify $\mathcal{V}_n$
with $\mathbb{R}^{n+1}$ using the terms from the binomial theorem to produce a
basis; for example, $\mathcal{V}_4 \to \mathbb{R}^5$ by
\begin{equation} 
v_{-4}\ x^4 + v_{-2}\ 4x^3y + v_{0}\ 6x^2y^2 
+ v_{2}\ 4xy^3 + v_{4}\ y^4 \mapsto (v_{-4}, v_{-2},
v_{0},v_{2},v_{4}) \in \mathbb{R}^{5}. \end{equation}

Let $M$ denote a 5-dimensional smooth manifold, and let $\mathcal{F}$ denote
the $\mathcal{V}_4$-valued coframe bundle over $M$, so the fiber
$\mathcal{F}_p$ is the set of all isomorphisms $\mathbf{T}_pM \to
\mathcal{V}_4$.  The coframe bundle $\mathcal{F}$ is a principal right
$GL(\mathcal{V}_4)$ bundle.

Recall that $\mathcal{V}_n$ is the unique irreducible representation of
$\mathfrak{sl}(2)$ of dimension $n{+}1$.  The action of $\mathfrak{sl}(2)$ on
$\mathcal{V}_n$ is generated by  
\begin{equation} X= y\frac{\partial}{\partial x},\ Y=
-x\frac{\partial}{\partial y},\ \text{and}\ H=
x\frac{\partial}{\partial x} - y \frac{\partial}{\partial y}.  \end{equation}
\begin{defn}
A $\boldsymbol{GL(2)}$\textbf{-structure} $\mathcal{B} \to M$ is a reduction of $\mathcal{F}$ with
fiber group $GL(2)\subset GL(\mathcal{V}_n)$ infinitesimally generated by $X$, $Y$,
$H$, and the scaling action $I$.  
\label{GL2def}
\end{defn}
Note that throughout this article, the symbols $GL(2)$ and $\mathfrak{gl}(2)$
are used to denote the abstract Lie group and its Lie algebra as well as any
particular irreducible representations $GL(2) \subset GL(\mathcal{V}_n)$ and
$\mathfrak{gl}(2) \subset \mathfrak{gl}(\mathcal{V}_n)$.
Since $SL(2)$ has a unique irreducible representation $\mathcal{V}_n$ in each
dimension $n{+}1$, it is always clear from context which representation of
$GL(2)$ is being considered.

Generally, a $GL(2)$-structure is said to have ``degree $n$'' if the base space
$M$ has dimension $n{+}1$.  The degree 3 case is thoroughly studied in
\cite{Bryant1991a}, and various observations are made for all $n$ in
\cite{Smith2009}.  This article considers only the degree 4 case.  The
$PGL(2)$ action generated by $X$, $Y$, and $H$ is the symmetry group of a
rational normal curve in $\mathbb{P}(\mathcal{V}_4)$.  A rational normal curve
is a curve of degree $n$ in $\mathbb{P}(\mathbb{R}^{n+1})$. All such curves are
$PGL({n+1})$-equivalent to $\{[g^n:g^{n-1}h:\cdots:gh^{n-1}:g^n]\}$
\cite{Harris1995}.  The de-projectivization of the rational normal curve is
the rational normal cone, which has a $GL(2)$ symmetry group and is usually
described as $\{(g^n,g^{n-1}h,\ldots,gh^{n-1},h^n)\}\subset \mathbb{R}^{n+1}$.
Thus, the geometric content of Definition~\ref{GL2def} is contained in the
following lemma.

\begin{lemma}
A $GL(2)$-structure $\mathcal{B} \to M^5$ is equivalent to a distribution of rational
normal cones $\mathbf{C} \subset \mathbf{T}M$.  For any $b \in \mathcal{B}_p$,
$b(\mathbf{C}_p) = \{ (gx+hy)^4:g,h\in \mathbb{R}\} \subset \mathcal{V}_4$.
\label{GL2cone} 
\end{lemma}

\begin{defn}[Integrability]
Let $\mathcal{B} \to M$ be a $GL(2)$-structure.
A $k$-dimensional linear subspace $E \subset \mathbf{T}_pM$ is
\textbf{k-secant} if $E \cap \mathbf{C}_p$ is $k$ distinct lines.
Equivalently, $E$ is k-secant if $E$ is spanned by $k$ vectors in $\mathbf{C}_p$.  
A $k$-dimensional submanifold $N \subset M$ is
\textbf{k-secant} if $\mathbf{T}_pN$ is a k-secant subspace of $\mathbf{T}_pM$
for every $p \in N$.
A $GL(2)$-structure is \textbf{k-integrable} if, for every k-secant linear
subspace $E \in Gr_k(\mathbf{T}M)$, there exists a $k$-secant submanifold $N$
with $E=\mathbf{T}_pN$.
\end{defn}
Locally, 1-integrability is uninteresting, as it is simply the existence of a
local flow for a vector field.  Since $M$ has dimension 5, being $4$-secant is
an open condition on $E \in Gr_4(\mathbf{T}M)$; however, the condition is
closed for $k{=}2$ (bi-secant) and $k{=}3$ (tri-secant).  A $GL(2)$-structure
that is both 2-integrable and 3-integrable is called ``2,3-integrable,'' or
simply ``integrable.''

Notice that all the definitions are projectively invariant.  Indeed, one could
equivalently study the principal right $PGL(2)$ bundle over $M$ defined by the
symmetries of rational normal curves in $\mathbb{P}\mathbf{T}M$.  When
considering the de-projectivized $GL(2)$ geometry, one occasionally
encounters an additional $\mathbb{Z}/2\mathbb{Z}$ symmetry by $\pm I$. For the
purposes of this article, it is easier to proceed using affine geometry and
deal with this symmetry when it arises.  

The geometric content of this article is a complete local description and
classification of $GL(2)$-structures of degree 4 that are 2,3-integrable.  Such
$GL(2)$-structures are particularly interesting because of their deep connection
to integrable PDEs.  The remainder of this section outlines how Ferapontov,
Hadjikos, and Khusnutdinova obtain 2,3-integrable $GL(2)$-structures from
certain hydrodynamic PDEs in \cite{Ferapontov2009}.

\subsection{Hessian Hydrodynamic Hyperbolic PDEs}
The space of 1-jets over $\mathbb{R}^3\to\mathbb{R}$ is $\mathbb{J}^1=
\{(\xi^i, z, p_i)\}$ where $1 \leq i,j \leq 3$.  The space of 2-jets is
$\mathbb{J}^2=  \{ (\xi^j, z, p_i, U_{ij} ) \}$ where $1\leq i,j \leq 3$ and
$U_{ij}=U_{ji}$.   The contact system on $\mathbb{J}^2$ is the differential
ideal $\mathcal{J}^2$ generated by  
\begin{equation}
\{ 
\mathrm{d}z - p_i \mathrm{d}\xi^i,\ 
\mathrm{d}p_i - U_{ij}\mathrm{d}\xi^j
\}. 
\label{contact}
\end{equation}
A second-order PDE on ${u:\mathbb{R}^3 \to \mathbb{R}}$ is equivalent to the level
set $F^{-1}(0)$ of some smooth function $F:\mathbb{J}^2\to \mathbb{R}$ on
which the projection $\mathbb{J}^2 \to \mathbb{J}^1$ is a submersion, and a
solution is a function $u:\mathbb{R}^3 \to \mathbb{R}$ whose jet-graph is a
subset of $F^{-1}(0)$.

Recall that a section of $\mathbb{J}^2$ is the jet-graph of a function if and
only if it is an integral of the contact system.
A contact transformation on $\mathbb{J}^2$ is a diffeomorphism
$\psi:\mathbb{J}^2 \to  \mathbb{J}^2$ such that $\psi^*(\mathcal{J}^2) =
\mathcal{J}^2$, and one studies PDEs geometrically by examining the properties
of $F^{-1}(0)$ that are invariant under contact transformations.

By a theorem of B\"acklund \cite[Theorem~4.32]{Olver1995}, every contact
transformation on $\mathbb{J}^2$ is the prolongation of a contact
transformation on $\mathbb{J}^1$.  Therefore, for any contact transformation
$\psi$, there exist functions $A_{ij}$,
$B_{ij}$, $C_{ij}$, $D_{ij}$, $a_i$, $m_i$, $\hat{a}_i$, $\hat{m}_i$, and $c$,
(written as four matrices, four column vectors, and a scalar) such that
\begin{equation}
\psi^*\begin{pmatrix}
\mathrm{d}\xi \\ \mathrm{d}z \\ \mathrm{d}p
\end{pmatrix}=
\begin{pmatrix}
B & \hat{m} & C \\
m^t & c & a^t \\
A & \hat{a} & D 
\end{pmatrix}\begin{pmatrix}
\mathrm{d}\xi \\ \mathrm{d}z \\ \mathrm{d}p
\end{pmatrix}.
\label{conformalcoords}
\end{equation}
The assumption that $\psi$ is a contact transformation implies 
that the symplectic form $\sigma = \mathrm{d}p_1\wedge\mathrm{d}\xi^1 +
\mathrm{d}p_2 \wedge\mathrm{d}\xi^2 + \mathrm{d}p_3\wedge\mathrm{d}\xi^3$ is
preserved up to scale, so $\psi^*(\sigma) = \lambda \sigma$, $\lambda \neq 0$.
One can compute $\tilde{U}_{ij} = \psi^*(U_{ij})$ by noting that
$\psi^*(\mathrm{d}p_i - U_{ij}\mathrm{d}\xi^j) \equiv 0$ modulo
$\mathcal{J}^2$, so in matrix notation
\begin{equation}
\begin{split}
\psi^*(\mathrm{d}p) - \tilde{U}\psi^*(\mathrm{d}\xi)
&= 
(A\mathrm{d}\xi + \hat{a}\mathrm{d}z + D \mathrm{d}p) -
\tilde{U}(B\mathrm{d}\xi + \hat{m}\mathrm{d}z + C \mathrm{d}p) \\
&\equiv
(A + \hat{a}p^t + DU)\mathrm{d}\xi - \tilde{U}(B + \hat{m}p^t +
CU)\mathrm{d}\xi.
\end{split}
\end{equation}
In particular, under the non-degeneracy assumption $\psi^*(\mathrm{d}\xi^1 \wedge
\mathrm{d}\xi^2 \wedge\mathrm{d}\xi^3) \neq 0$, this yields 
\begin{equation}\tilde{U} = (A +
\hat{a} p^t + DU)(B+\hat{m}p^t + CU)^{-1}.\end{equation}  

For this article, consider only PDEs on $u:\mathbb{R}^3 \to \mathbb{R}$ of the
form
\begin{equation}
F(u_{11}, u_{12}, u_{13}, u_{22}, u_{23}, u_{33}) =0.
\label{PDE}
\end{equation}
Instead of considering the orbit of this PDE under all contact
transformations, consider only those contact transformations that preserve the
class of Hessian-only equations, those of the form in Equation~(\ref{PDE}).   That is, consider
only those contact transformations $\psi$ such that $\psi^*(\mathrm{d}U_{ij})
\equiv 0$ modulo $\{\mathrm{d}U_{kl}\}$.  Then it must be that $\hat{a}=\hat{m}=0$ in
Equation~(\ref{conformalcoords}).
Thus, the contact transformations that preserve the Hessian-only form of
Equation~(\ref{PDE}) are elements of the conformal symplectic group, 
\begin{equation}
\begin{split}
CSp(3) 
&= \{ g \in GL(6,\mathbb{R}) ~:~ \sigma(gv,gw) = \lambda\sigma(v,w)\ \forall
v,w\ \text{(any $\lambda\in \mathbb{R}^*$)}\}\\
&= \left\{ g= \begin{pmatrix}
B & C \\ 
A & D
\end{pmatrix} ~:~ 
0=A^tB - B^tA =D^tC-C^tD, \lambda I_3=D^tB - C^tA
\right\}.
\end{split}
\label{Sp3}
\end{equation}
In the most general case, the completed domain of all Hessian-only $F$'s is the Lagrangian Grassmannian,
\begin{equation}
\Lambda = \mathbb{J}^2/\mathbb{J}^1 = \{ U \in Gr_3(\mathbb{R}^6) ~:~ \sigma|_U = 0 \},
\end{equation}
with local coordinates given, for example, by its non-$\xi$-vertical open subset
\begin{equation}
\Lambda^o = \{ U \in \Lambda :
\mathrm{d}\xi^1\wedge\mathrm{d}\xi^2\wedge\mathrm{d}\xi^3|_U \neq 0 \}.
\end{equation}

The conformal symplectic group forms a bundle, $\Pi:CSp(3) \to \Lambda$.  In
fact, when considering local transformations of PDEs, only the transformations in $CSp(3)^o =
\Pi^{-1}(\Lambda^o)$ are permissible, as the condition
$\mathrm{d}\xi^1\wedge\mathrm{d}\xi^2\wedge\mathrm{d}\xi^3 \neq 0$ must be
preserved.  As above, one could eliminate the conformal factor and consider
only $Sp(3)^o$ actions by projectivizing, but there is little utility in doing
so.

To be explicit regarding the bundle, suppose $U \in \Lambda^o$.
The open manifold $\Lambda^o$ has coordinates given by the components of symmetric matrices,
$U_{ij}=U_{ji}$, and $\mathrm{d}p_i \equiv U_{ij}
\mathrm{d}\xi^j$ modulo $\mathcal{J}^2$. 
Then 
\begin{equation}
\psi^*\begin{pmatrix}
\mathrm{d} \xi\\ \mathrm{d}p
\end{pmatrix} =
\begin{pmatrix}
B & C \\
A & D
\end{pmatrix}
\begin{pmatrix}
\mathrm{d} \xi\\ \mathrm{d}p
\end{pmatrix}.
\label{Sp3action}
\end{equation}
The assumption that $\mathrm{d}\xi$ and $\psi^*(\mathrm{d}\xi)$ both have maximum
rank when pulled back to $N$ implies that $(B+CU)$ is nonsingular.  Moreover,
Equation~(\ref{Sp3action}) shows that 
$CSp(3)^o$ acts on $\Lambda^o$ by 
\begin{equation}
g:U \mapsto g(U) = (A+DU)(B+CU)^{-1}.
\label{Sp3actionU}
\end{equation}
The fiber subgroup over $U \in \Lambda^o$ is the stabilizer of $U$, so
\begin{equation}
\Pi^{-1}(U) = \{ g : U = g(U)=(A+DU)(B+CU)^{-1}\} \cong \{ g :
A=0\}.\label{Sp3Fiber}\end{equation}
Because $I_6$ is in the fiber over $0 \in \Lambda^o$, the projection
$\Pi:CSp(3)^o \to \Lambda^o$ can be computed from Equation~(\ref{Sp3actionU}) as
\begin{equation} \Pi:g \mapsto g(0) = AB^{-1}. \label{Sp3U}
\end{equation}

The most important fact about Equation~(\ref{PDE}) is that $M=F^{-1}(0)$
admits a natural $GL(2)$-structure wherever the PDE is hyperbolic.  Recall
that a second-order PDE is called hyperbolic if its leading symbol is a
non-degenerate matrix with split signature, and this property is invariant
under contact transformations.
A related notion of hyperbolicity for hyperplanes sitting in
projective space is needed to describe the induced $GL(2)$-structure.
The complex Veronese variety is the 2-dimensional projective variety
\begin{equation} \{ [Z_1Z_1: Z_1Z_2 : Z_1Z_3:  Z_2Z_3: Z_3Z_3], 
Z\in \mathbb{CP}^2 \} \subset \mathbb{CP}^5.  \end{equation} The intersection
of a generic hyperplane with the Veronese variety is a 1-dimensional
rational normal curve, and this curve is uniquely given by such an
intersection.  The real case needed for the present PDE theory
requires a bit more detail to describe accurately.  Consider $\mathbb{RP}^5$
with coordinates $[W_1: \cdots: W_6]$, $\mathbb{RP}^3$ with coordinates $[Z_1:
Z_2: Z_3]$, and the Veronese variety defined over $\mathbb{R}$ as above.  A
generic hyperplane in $\mathbb{RP}^5$ is defined uniquely (up to scale and sign) by a
single equation 
\begin{equation}a_{11}W_1 + a_{12}W_2 + a_{13}W_3 + a_{22}W_4 + a_{23} W_5  +
a_{33}W_6 = 0.\label{hypcut}\end{equation}The intersection of this hyperplane with the Veronese variety
yields an equation on $\mathbb{R}^3$: \begin{equation} a_{11} (Z_1)^2 + a_{12}
Z_1Z_2 + a_{13}Z_1Z_3 + a_{22}(Z_2)^2 + a_{23}Z_2Z_3 + a_{33}(Z_3)^2 = 0.
\label{quadric} \end{equation} Depending on the (real) coefficients $a_{ij}$,
Equation~(\ref{quadric}) may or may not have real solutions $Z$.
If Equation~(\ref{quadric}) has real solutions, then the de-projectivized
solution in $\mathbb{R}^3$ is a real quadric surface.  This quadric may or may
not be degenerate.  The existence of real nondegenerate solutions is an open
condition on the hyperplane in the topology of $Gr_5(\mathbb{R}^6)$.  If this
condition is satisfied, the hyperplane defined by $\{a_{ij}\}$ is called
hyperbolic.  The precise algebraic condition for hyperbolicity is that the
real symmetric matrix $(a_{ij})$ is nonsingular and has split signature
\cite{CRC}.  Using the identifications
$\mathbb{R}^6 = Gr_5(\mathbb{R}^6) = \mathrm{Sym}^2(\mathbb{R}^3)$, the cone over
the Veronese variety (Veronese cone) is identified with the space of symmetric
matrices $(a_{ij})$ such that $\mathop{\mathrm{rank}}(a_{ij})\leq 1$.
A hyperbolic hyperplane in $\mathbb{R}^6$ intersects the Veronese cone in a
rational normal cone, and every rational normal cone in $\mathbb{R}^6$ can be
written (uniquely, up to scale) this way \cite{Harris1995}. 

\begin{lemma}
Let $F:\Lambda^o \to \mathbb{R}$ be a smooth function, and suppose $U \in
\Lambda^o$ such that $F(U)=0$, $\mathrm{d}F_U\neq 0$, and
$\ker(\mathrm{d}F_U)$ is hyperbolic as a hyperplane in
$\mathbf{T}_U\Lambda^o~=~\mathrm{Sym}^2(\mathbb{R}^3)$.  Then there is an
open 5-dimensional submanifold $M \subset \Lambda^o$ defined by $F|_M =0$ in a
neighborhood of $U$, and $M$ admits a distribution $\mathbf{C}$ of rational
normal cones.  That is, $M$ admits a $GL(2)$-structure.
\label{hyperbolics}
\end{lemma}

\begin{proof}
Aside from the condition on hyperbolicity of the tangent space, the lemma is
simply a statement of the implicit function theorem.   Here is an explanation
of hyperbolicity and its relation to $GL(2)$-structures.

Fix a coframe $s$ on $\Lambda^o$ such that $s_U:\mathbf{T}_U \Lambda^o \to
\mathrm{Sym}^2(\mathbb{R}^n)$.  (An obvious choice is $s = \mathrm{d}U$, using
the symmetric matrices as coordinates for $\Lambda^o$.)  So, there is a 
distribution of 3-dimensional Veronese cones defined by $\{ P \in \mathbf{T}\Lambda^o :
\mathop{\mathrm{rank}}s(P) \leq 1 \}$.  This distribution of cones is
$CSp(3)$-invariant, for in the first part of \cite[Theorem XX]{Cartan1909},
Cartan proves that the above bundle $CSp(3)$ is exactly the frame bundle over
$\Lambda$ whose action on $\mathbf{T}_U \Lambda \simeq \mathrm{Sym}^2(\mathbb{R}^n)$ is
to act irreducibly on the coefficients of the equation of a cone, for example
Equation~(\ref{quadric}).

In the context of the lemma, Equation~(\ref{hypcut}) describes the
intersection of $\ker(\mathrm{d}F_U)$ with the Veronese cone, so the symmetric
matrix $(a_{ij})$ is, up to scale, the leading symbol of $F$ at $U$.  Therefore,
if $U$ is a regular point for the regular value $0$ of ${F:\Lambda^o \to
\mathbb{R}}$ and if the PDE ${F=0}$ is hyperbolic at $U$, then $F^{-1}(0)$
admits a field of rational normal cones near $U$ given by these intersections.
\end{proof}

The fact that every hyperbolic PDE of the form in Equation~(\ref{PDE}) admits
a $GL(2)$-structure would only be an algebraic curiosity, except that the
integrability of the PDE is intimately related to the 2,3-integrability of the
$GL(2)$-structure.

\begin{thm}[Theorem 3 in \cite{Ferapontov2009}]
Fix a hyperbolic PDE of the form in Equation~(\ref{PDE}) and its corresponding
$GL(2)$-structure ${\pi:\mathcal{B} \to M}$.  Then $\mathcal{B}$ is 2-integrable.
Moreover, the PDE is integrable via 3-parameter hydrodynamic reductions if and
only if $\mathcal{B}$ is also 3-integrable.  \label{thm:F23int} 
\end{thm}

Theorem~\ref{thm:F23int} as presented in \cite{Ferapontov2009} apparently
requires that, for each 3-secant $N^3 \subset M^5$, the characteristic net
defined by the intersection $\mathbf{T}N \cap \mathbf{C}$ is a coordinate net.
According to Corollary~\ref{holonomic-net}, this requirement is redundant.

As discussed in \cite{Ferapontov2009}, for PDEs in three independent
variables, the existence of 3-parameter hydrodynamic reductions parametrized
by three functions of one variable implies the existence of $k$-parameter
hydrodynamic reductions parametrized by $k$ functions of one variable for
all $k \geq 3$. 

\begin{cor} 
$F(u_{ij})=0$ is integrable via $k$-parameter hydrodynamic
reductions for all $k\geq2$ if and only if the induced $GL(2)$-structure over
$F^{-1}(0)$ is 2,3-integrable.  
\end{cor}

\begin{defn}[Hessian hydrodynamic PDEs]
A PDE on $u:\mathbb{R}^3 \to \mathbb{R}$ is \textbf{Hessian hydrodynamic} if and only if it is hyperbolic,
is of the form $F(u_{11}, \ldots, u_{33})=0$, and is integrable by means of
3-parameter hydrodynamic reductions.
\end{defn}
The remainder of this article is a study of Hessian hydrodynamic PDEs via
their associated 2,3-integrable $GL(2)$-structures.

\section{A Preferred Connection}
\label{connection}
In this section, Cartan's method of equivalence is applied to
$GL(2)$-structures.  Cartan's method of equivalence is a standard tool in the
field of exterior differential systems; it is an algorithm for finding all of
the differential invariants of a geometric structure, and the first step is to 
fix a preferred connection among all the possible pseudo-connections of a
geometric structure \cite{Bryant2003, Gardner1989, Ivey2003}.  (The
distinction between a connection and a pseudo-connection is simply whether it
was obtained via such an algorithm.)  The result of the algorithm
is a preferred global coframe for the $GL(2)$-structure that splits into the
semi-basic ``tautological'' form, $\omega$, and the vertical
$\mathfrak{gl}(2)$-valued ``connection'' form, $\theta$.

For a $GL(2)$-structure $\pi:\mathcal{B} \to M$, let $\omega$ denote the
tautological 1-form defined by $\omega_b = b \circ \pi_*:
\mathbf{T}_b\mathcal{B} \to \mathcal{V}_4$.  As always, $\omega$ is a
globally-defined 1-form on $\mathcal{B}$ with linearly independent components, and it is
semi-basic, meaning $\omega|_{\ker \pi_*}=0$.  A pseudo-connection on
$\mathcal{B}$ is a 1-form $\theta$ taking values in the non-trivial
representation of $\mathfrak{gl}(2)$ in $\mathfrak{gl}(\mathcal{V}_4)$
such that $\mathrm{d}\omega = - \theta\wedge\omega + T(\omega\wedge\omega)$
for some torsion $T(b):\mathcal{V}_4 \wedge \mathcal{V}_4 \to \mathcal{V}_4$. 

The goal is to fix a particular $\theta$ that
minimizes the torsion $T$.  To carry out this process, one
needs to understand the $SL(2)$ representation theory on $\mathcal{V}_n$.  The
decomposition of the tensor product ${\mathcal{V}_m \otimes \mathcal{V}_n}$ into
irreducible components is
\begin{equation} \mathcal{V}_m \otimes \mathcal{V}_n = \mathcal{V}_{|m-n|}
\oplus \mathcal{V}_{|m-n|+2} \oplus \cdots \oplus \mathcal{V}_{m+n-2} \oplus
\mathcal{V}_{m+n}. \end{equation} 
The projections onto the various components are given by 
the Clebsch--Gordan \cite{Bryant1991a,Humphreys1972} pairings
$\pair{\cdot,\cdot}_p:\mathcal{V}_m \otimes \mathcal{V}_n \to
\mathcal{V}_{m+n-2p}$, which are provided by the formula
\begin{equation} \pair{u,v}_p = \frac{1}{p!} \sum_{k=0}^p (-1)^k \binom{p}{k}
\frac{\partial^p u}{\partial x^{p-k} \partial y^k}\cdot \frac{\partial^p
v}{\partial x^{k} \partial y^{p-k}}.\end{equation}
This pairing has some important properties.  Notice that $\pair{u,v}_p =
(-1)^p\pair{v,u}_p$ and that the pairing is nontrivial for $0 \leq p \leq
\min\{m,n\}$.  Hence, the tensor decomposition can be further refined in terms
of the symmetric and alternating tensors: 
\begin{equation} 
\begin{split}
\mathcal{V}_n \circ \mathcal{V}_n &= \mathcal{V}_{2n} \oplus
\mathcal{V}_{2n-4} \oplus \cdots \oplus \mathcal{V}_{0 \text{ or } 2},\\
\mathcal{V}_n \wedge \mathcal{V}_n &= \mathcal{V}_{2n-2}
\oplus \mathcal{V}_{2n-6} \oplus \cdots \oplus \mathcal{V}_{2 \text{ or } 0}.
\end{split}
\end{equation}
Notice too that $\pair{\cdot, \cdot}_n: \mathcal{V}_n \otimes \mathcal{V}_n
\to \mathcal{V}_0 = \mathbb{R}$ is a non-degenerate symmetric  or
skew bilinear form.  Hence, for fixed $u \in \mathcal{V}_n$ the map
$\pair{u,\cdot}_n:\mathcal{V}_n \to \mathcal{V}_{0}=\mathbb{R}^1$ provides 
a natural identification, $\mathcal{V}_n = \mathcal{V}_n^*$, and one need never
distinguish between dual spaces when considering representations.

For any derivation over $\mathbb{R}[x,y]$, a Leibniz rule over the pairing
holds.  Because $SL(2)$ is infinitesimally generated by $X$, $Y$, and $H$,
this means that the pairings are $SL(2)$-equivariant.  That is,
$\alpha(\pair{u,v}_p) = \pair{\alpha(u),v}_p + \pair{u,\alpha(v)}_p$ for any
$\alpha \in \mathfrak{sl}(2)$, which implies that $a\cdot \pair{u,v}_p =\pair{a\cdot u ,a
\cdot v }_p$ for any $a \in SL(2)$.  The pairing is not $GL(2)$-equivariant,
but the scaling action is easily computed easily where required. The geometric
objects encountered here are projectively defined, so this variance in scaling
is of little concern. 

Most importantly, the pairing can be generalized to binary-polynomial-valued
differential forms over a manifold.  If $u \in \Gamma(\wedge^r
\mathbf{T}^*\mathcal{B} \otimes \mathcal{V}_m)$ and $v \in \Gamma(\wedge^s
\mathbf{T}^*\mathcal{B} \otimes \mathcal{V}_n)$, then extend the definition by
using the wedge product: \begin{equation} \pair{u,v}_p = \frac{1}{p!}
\sum_{k=0}^p (-1)^k \binom{p}{k} \frac{\partial^p u}{\partial x^{p-k} \partial
y^k}\wedge \frac{\partial^p v}{\partial x^{k} \partial y^{p-k}}.\end{equation}
In this generalization, the symmetry of the pairing is further altered by the
degree of the forms: $\pair{u,v}_p = (-1)^{rs+p}\pair{v,u}_p$.  If $\lambda$
is an $\mathbb{R}$-valued $1$-form acting by the scaling action $\lambda I_m$
on $\mathcal{V}_m$, then $\lambda\wedge u$ may be written as
the trivial pairing $\pair{\lambda,u}_0 = (-1)^r\pair{u,\lambda}_0 \in
\Gamma(\wedge^{r+1} \mathbf{T}^*\mathcal{B} \otimes \mathcal{V}_m)$.

As $SL(2)$ representations, $\mathfrak{gl}(2) =
\mathfrak{sl}(2) \oplus \mathbb{R} = \mathcal{V}_2 \oplus \mathcal{V}_0$.
Thus, a $\mathfrak{gl}(2)$-valued pseudo-connection $\theta$ decomposes as
$(\varphi, \lambda)$ with $\varphi \in \Gamma(\mathbf{T}^*\mathcal{B} \otimes
\mathcal{V}_2)$ and $\lambda \in \Gamma(\mathbf{T}^*\mathcal{B} \otimes
\mathcal{V}_0)$.  The torsion of a generic $\theta$ is 
\begin{equation}\begin{split}
T :\mathcal{B}\to \mathcal{V}_4 \otimes (\mathcal{V}_4^* \wedge \mathcal{V}_4^*)
&= 
\mathcal{V}_4 \otimes (\mathcal{V}_2 \oplus \mathcal{V}_6) \\
&= 
(\mathcal{V}_2 \oplus \mathcal{V}_4 \oplus \mathcal{V}_6) 
\oplus
(\mathcal{V}_2 \oplus \mathcal{V}_4 \oplus \mathcal{V}_6 \oplus
\mathcal{V}_8\oplus \mathcal{V}_{10}),\ \text{so}\\
T &= (T^2_2 + T^2_4 + T^2_6) + (T^6_2 + T^6_4 + T^6_6 + T^6_8 + T^6_{10}).
\end{split}
\label{torsion}
\end{equation}
In this notation, each $T^r_n$ is a distinct irreducible component of $T$. The
lower index $n$ indicates the weight of the representation in which $T^r_n$
takes values, and the upper index $r$ indicates the summand from which it was
obtained.

Thus, Cartan's first structure equation for a $GL(2)$-structure may be written in
either vector form or polynomial form:
\begin{equation}
\begin{split}
\mathrm{d}\omega &= 
-\theta\wedge\omega + T(\omega\wedge\omega)\\
&= 
-\pair{\varphi,\omega}_1 - \pair{\lambda,\omega}_0 + 
\pair{ T^2_{ 2}, \pair{\omega,\omega}_3}_0 +
\pair{ T^2_{ 4}, \pair{\omega,\omega}_3}_1 +
\pair{ T^2_{ 6}, \pair{\omega,\omega}_3}_2 \\
&\phantom{=} +
\pair{ T^6_{ 2}, \pair{\omega,\omega}_1}_2 +
\pair{ T^6_{ 4}, \pair{\omega,\omega}_1}_3 +
\pair{ T^6_{ 6}, \pair{\omega,\omega}_1}_4 \\
&\phantom{=} +
\pair{ T^6_{ 8}, \pair{\omega,\omega}_1}_5 +
\pair{ T^6_{10}, \pair{\omega,\omega}_1}_6.
\end{split}
\label{structure}
\end{equation}
Explicitly, the connection term is 
\begin{equation}
\theta\wedge\omega =
\begin{pmatrix}
-8\varphi_{0}+\lambda & 8\varphi_{-2} &0 & 0 & 0 \\ -2\varphi_{2} &
-4\varphi_{0}+\lambda & 6\varphi_{-2} & 0 & 0 \\ 0 & -4\varphi_{2} &  \lambda &
4\varphi_{-2} & 0 \\ 0 & 0 &-6\varphi_{2} &  4\varphi_{0}+\lambda &
2\varphi_{-2} \\ 0 & 0 & 0 &-8\varphi_{2} &  8\varphi_{0}+\lambda \\
\end{pmatrix} \wedge \begin{pmatrix} \omega^{-4}\\ \omega^{-2}\\ \omega^{0}\\
\omega^{2}\\ \omega^{4} \end{pmatrix}, 
\label{structurematrix}
\end{equation}
so the matrix representation of $\theta$ is $(2\varphi_{-2}{X} -
2\varphi_{0}{H}+2\varphi_{2}{Y} + \lambda{I}_5)$, which
takes values in $\mathfrak{gl}(2) \subset \mathfrak{gl}(\mathcal{V}_4)$.

\begin{thm}
A generic $GL(2)$-structure ${\pi:\mathcal{B}\to M}$ admits a two-dimensional
family of connections such that the essential torsion $T$ decomposes
irreducibly as 
\begin{equation}T= T_2 + T_6 + T_8 + T_{10} \in \mathcal{V}_2 \oplus
\mathcal{V}_6 \oplus \mathcal{V}_8 \oplus \mathcal{V}_{10} \subset
\mathcal{V}_4 \otimes (\mathcal{V}_4^* \wedge \mathcal{V}_4^*).\end{equation} 
The two-dimensional family is parametrized by the possible equivariant
inclusions of 
$\mathcal{V}_2 \oplus \mathcal{V}_6 \oplus \mathcal{V}_8 \oplus
\mathcal{V}_{10}$ into $\mathcal{V}_4 \otimes (\mathcal{V}_4^* \wedge
\mathcal{V}_4^*)$, and once such an inclusion is chosen, the connection is
unique.
\label{unique-connection}
\end{thm}

\begin{proof}
Changes of pseudo-connection are of the form $\hat{\varphi}=\varphi+P(\omega)$
and $\hat{\lambda}=\lambda+Q(\omega)$ where $P \in \mathcal{V}_2 \otimes
\mathcal{V}_4^* = \mathcal{V}_2 \oplus \mathcal{V}_4 \oplus \mathcal{V}_6$ and
$Q \in \mathcal{V}_0 \otimes \mathcal{V}_4^* = \mathcal{V}_4$.  Preferred
connections are obtained by analyzing the skewing map $\delta$ that describes
how changes of pseudo-connection (equivalently, changes of horizontal section
of the frame bundle) affect the torsion \cite{Bryant2003}:  
\begin{equation}
\begin{diagram}
\dgARROWLENGTH=0.5\dgARROWLENGTH
\node{0 \to\mathfrak{gl}(2)^{(1)}} \arrow{e} 
\node{(\mathcal{V}_2 \oplus \mathcal{V}_0) \otimes \mathcal{V}_4^*} 
\arrow{e,t}{\delta} \node{\mathcal{V}_4 \otimes(\wedge^2 \mathcal{V}_4^*)} 
\arrow{e} \node{H^{0,2}(\mathfrak{gl}(2)) \to 0.}\\
\node[3]{\mathcal{B}}\arrow{nw,b}{P,Q}\arrow{n,l}{T}\arrow{ne,r}{[T]}
\end{diagram}
\label{sequence}
\end{equation}
For current purposes, $\mathfrak{gl}(2)^{(1)}$ and $H^{0,2}(\mathfrak{gl}(2))$
are defined by the exactness of the sequence \cite{Gardner1989}.   They depend
on the representation of the group $GL(2)$, in this case the irreducible
representation $\mathcal{V}_4$. To find the
space of essential torsion, $H^{0,2}(\mathfrak{gl}(2))$, one must compute
$\delta P$ and $\delta Q$.

Fix $P \in \mathcal{V}_2 \oplus \mathcal{V}_4 \oplus \mathcal{V}_6$, where
$\mathcal{V}_2 \owns P(\omega) = \pair{P_2, \omega}_2 + \pair{P_4,\omega}_3 +
\pair{P_6,\omega}_4$. Let $\delta P \in \mathcal{V}_4 \otimes
\wedge^2(\mathcal{V}_4)$ have components $\delta P = \delta P^2_2 + \delta
P^2_4 + \delta P^2_6 + \delta P^6_2 + \delta P^6_4 + \delta P^6_6 + \delta
P^6_8 + \delta P^6_{10}$, similar to the decomposition of $T$.  Since these
are irreducible components and the action of $\delta$ must be
$SL(2)$-equivariant, Schur's lemma implies $\delta$ must preserve the weights
of the representations.
In particular there must exist constants $a_2$, $a_4$, $a_6$, $b_2$, $b_4$,
and $b_6$ such that $\delta P^2_2 = a_2 P_2$, $\delta P^6_2 = b_2 P_2$, and so on.  Thus,
\begin{equation}
\begin{split}
0 &= \pair{ P(\omega),\omega}_1 - \delta P(\omega,\omega)\\
&=
 \pair{ \pair{P_2,\omega}_2,\omega}_1 
+\pair{ \pair{P_4,\omega}_3,\omega}_1 
+\pair{ \pair{P_6,\omega}_4,\omega}_1 \\
&\ - 
\pair{a_2 P_2,\pair{\omega,\omega}_3}_0 - 
\pair{a_4 P_4,\pair{\omega,\omega}_3}_1 - 
\pair{a_6 P_6,\pair{\omega,\omega}_3}_2 \\
&\ - 
\pair{b_2 P_2,\pair{\omega,\omega}_1}_2 - 
\pair{b_4 P_4,\pair{\omega,\omega}_1}_3 - 
\pair{b_6 P_6,\pair{\omega,\omega}_1}_4.
\end{split}
\label{skewP}
\end{equation}
Carrying out this computation shows that 
\begin{equation}
a_2=\frac3{10},\  b_2=\frac15,\ 
a_4=\frac12,\  b_4=0,\ 
a_6=-\frac15,\ b_6=-\frac1{20}.
\label{image of skew P}
\end{equation}

Similarly, fix $Q \in \mathcal{V}_4$, where
$\mathcal{V}_0 \owns Q(\omega) = \pair{Q_4, \omega}_4$.  Let $\delta Q \in
\mathcal{V}_4 \otimes \wedge^2(\mathcal{V}_4)$ have components $\delta Q =
\delta Q^2_2 + \delta Q^2_4 + \delta Q^2_6 + \delta Q^6_2 + \delta Q^6_4 +
\delta Q^6_6 + \delta Q^6_8 + \delta Q^6_{10}$, but again the image must have
the same weight as the domain.
In particular there must exist constants $c_4$, and $d_4$ such that
\begin{equation}
\begin{split}
0 &= \pair{ Q(\omega),\omega}_0 - \delta Q(\omega,\omega) \\
&=
\pair{ \pair{Q_4,\omega}_4,\omega}_0 
- \pair{c_4 Q_4,\pair{\omega,\omega}_3}_1 
- \pair{d_4 Q_4,\pair{\omega,\omega}_1}_3.
\end{split}
\label{skewQ}
\end{equation}
Carrying out this computation shows that 
\begin{equation}
c_4=-\frac1{40},\ d_4=-\frac1{160}.
\label{image of skew Q}
\end{equation}

Fix a generic pseudo-connection $(\varphi,\lambda)$ with torsion
$T$.  Consider another pseudo-connection $\hat{\varphi} = \varphi+P(\omega)$,
$\hat{\lambda}=\lambda+Q(\lambda)$ with torsion $\hat{T}$. Then
\begin{equation}\begin{split}
\hat{T}(\omega,\omega)
&=
\mathrm{d}\omega + \pair{\hat\varphi,\omega}_1 + \pair{\hat\lambda,\omega}_0\\
&= 
\mathrm{d}\omega + \pair{\varphi,\omega}_1 + \pair{P(\omega),\omega}_1 +
\pair{Q(\omega),\omega}_0\\
&=
(T+\delta P+\delta Q)(\omega,\omega).
\end{split}\end{equation}
Using Equation~(\ref{image of skew P}) and Equation~(\ref{image of skew Q}), the
absorption of torsion is dictated by the solvability of the following
equations in terms of $P$ and $Q$ for fixed $T$ and $\hat{T}$:
\begin{equation}
\begin{split}
\hat{T}^2_{ 2} &= T^2_{ 2} + \frac{3}{10}P_{ 2},\\
\hat{T}^2_{ 4} &= T^2_{ 4} + \frac{1}{2}P_{ 4} - \frac{1}{40} Q_{ 4},\\
\hat{T}^2_{ 6} &= T^2_{ 6} - \frac{1}{5}P_{ 6},\\
\hat{T}^6_{ 2} &= T^6_{ 2} + \frac{1}{5}P_{ 2},\\
\hat{T}^6_{ 4} &= T^6_{ 4} - \frac{1}{160} Q_{ 4},\\
\hat{T}^6_{ 6} &= T^6_{ 6} - \frac{1}{20}P_{ 6},\\
\hat{T}^6_{ 8} &= T^6_{ 8},\\
\hat{T}^6_{10} &= T^6_{10}.
\end{split}
\end{equation}
Generally, one may choose $P_2$ to force exactly one linear combination of
$\hat{T}^6_2$ and $\hat{T}^2_2$ to vanish.  Similarly, one may choose
$P_6$ to force exactly one linear combination of $\hat{T}^6_6$ and $\hat{T}^2_6$ to vanish.
Unique $Q_4$ and $P_4$ eliminate $\hat{T}^6_4$ and
$\hat{T}^2_4$.  All other components of $\hat{T}$ are fixed. 
Thus, 
$\mathfrak{gl}(2)^{(1)}=\mathfrak{sl}(2)^{(1)}=0$ and
$H^{0,2}(\mathfrak{sl}(2)) = \mathcal{V}_2 \oplus
\mathcal{V}_4 \oplus \mathcal{V}_6 \oplus \mathcal{V}_8 \oplus
\mathcal{V}_{10}$, while  $H^{0,2}(\mathfrak{gl}(2)) = \mathcal{V}_2 \oplus
\mathcal{V}_6 \oplus \mathcal{V}_8 \oplus
\mathcal{V}_{10}$.
\end{proof}

\begin{cor}[Preferred GL(2) Connection]
A $GL(2)$-structure $\mathcal{B} \to M$ admits a unique connection $\varphi,\lambda$ such
that $\mathcal{B}$ has first structure equation
\begin{equation}
\begin{split}
\mathrm{d}\omega &= - \pair{\varphi,\omega}_1 - \pair{\lambda,\omega}_0  \\
&\phantom{=}+
\pair{T_2,\pair{\omega,\omega}_1}_2 + 
\pair{T_6,\pair{\omega,\omega}_1}_4\\
&\phantom{=}+
\pair{T_8,\pair{\omega,\omega}_1}_5 +
\pair{T_{10},\pair{\omega,\omega}_1}_6.
\end{split}
\end{equation}
\label{coframe}
\end{cor}
\begin{proof}
Of the possible connections, choose the one that absorbs $T^2_{2}$ and
$T^2_{6}$.  The remaining essential torsion is $T=T^6_2 + T^6_6 + T^6_8 +
T^6_{10}$.
\end{proof}

Henceforth, all references to $\theta$, $\varphi$, $\lambda$, and $T$ assume
the connection in Corollary~\ref{coframe}. The specification of this
connection over the others is arbitrary, but it does not affect any subsequent
theorems in this article, since the $\mathcal{V}_2$ and $\mathcal{V}_6$
components of essential torsion turn out to be unimportant in the study of integrable
PDEs.

\section{2-Integrability}
\label{2int}

\begin{thm} 
If a $GL(2)$-structure ${\pi:\mathcal{B}\to M}$ is 2-integrable, then $T=T_8$
(that is, $T_2=T_6=T_{10}=0$), and the bi-secant surfaces in $M$ are parametrized
by two functions of one variable.  Conversely, if $\mathcal{B}$ is a smooth
$GL(2)$-structure with $T=T_8$, then $\mathcal{B}$ is 2-integrable.
\label{bisecant} 
\end{thm}

\begin{proof}
To prove the theorem, one must find the conditions on $\mathcal{B}$ that allow
an arbitrary bi-secant plane $E \in Gr_2(\mathbf{T}_pM)$ to be extended to a
bi-secant surface $N \subset M$. 

In a neighborhood $M'$ of $p$, Fix $u:M' \to \mathcal{B}(M')$, a section of
$\pi:\mathcal{B}\to M$; that is, fix $u$, a $GL(2)$ coframe on $M'$.  Let
$b=u(p)$.  Since $E$ is bi-secant, $b(E) \subset \mathcal{V}_4$ is
spanned by $(g_1x-h_1y)^4$ and $(g_2x-h_2y)^4$ with $g_1h_2  \neq g_2h_1$.
Through a $GL(2)$ frame adaptation redefining $u$, one may assume that $b(E) =
\mathop{\mathrm{span}} \{x^4, y^4\} \subset \mathcal{V}_4$.  Let $\tilde{E} = u_*(E) \in
Gr_2(\mathbf{T}_{b}\mathcal{B})$.  Then
$\omega^{-4}\wedge\omega^{4}|_{\tilde{E}} \neq 0$.

Consider the linear Pfaffian exterior differential system $\mathcal{I}$
differentially generated by the $1$-forms $\{ \omega^{-2}, \omega^0, \omega^{2}\}$ with
independence condition the $\omega^{-4}\wedge\omega^4\neq0$.  It suffices to
prove the existence of a surface $\tilde{N} \subset \mathcal{B}$ that is
integral to $\mathcal{I}$, because
$u(\pi_*(\mathbf{T}\tilde{N}))=\omega_u(\mathbf{T}\tilde{N}) =
\mathop{\mathrm{span}}\{x^4,  y^4\}$ implies that the surface
$N=\pi(\tilde{N})\subset M'$ is bi-secant.  In fact, by adapting the moving
coframe $u$ appropriately, every bi-secant surface through $E$ must arise
this way.

The generating 2-forms of $\mathcal{I}$ are 
\begin{equation}
\mathrm{d}\begin{pmatrix}
\omega^{-2} \\ \omega^{0}\\ \omega^{2}
\end{pmatrix} \equiv
\begin{pmatrix}
-2\varphi_2  & 0 \\
0           & 0 \\
0           & 2\varphi_{-2}  \\
\end{pmatrix}
\wedge
\begin{pmatrix}
\omega^{-4}\\
\omega^{ 4} 
\end{pmatrix} + 
\begin{pmatrix}
\tau^{-2}\\\tau^{0}\\\tau^{2}\end{pmatrix}\omega^{-4}\wedge\omega^4
\label{involutive structure}
\end{equation}
modulo $\omega^{-2}, \omega^{0}, \omega^{2}$, where
\begin{equation}
\begin{split}
\tau^{-2} &= 
48\ T_{2,-2} + 8640\ T_{6,-2} + 322560\ T_{8,-2} - 4838400\ T_{10,-2}\\
\tau^{0} &=
-96\ T_{2,0} + 23040\ T_{6,0} - 4838400\ T_{10,0} \\
\tau^{2} &=
48\ T_{2,2} + 8640\ T_{6,2} - 322560\ T_{8,2} - 4838400\ T_{10,2}
\end{split}
\label{involutive torsion}
\end{equation}

Because of the independence condition, integral elements exist
only when the torsion can be absorbed.  The torsion component $\tau^{0}$ can
never be absorbed, so integral manifolds exist only when $\tau^{0} = 0$.  The
condition of $2$-integrability means that every $2$-secant plane is tangent to
a $2$-secant surface, but the $GL(2)$ action is transitive on $2$-secant planes in
$\mathbf{T}_pM$; therefore, 2-integrability implies that $\tau^{0}=0$ for every
element in the $GL(2)$ orbit of $T$.  Under the $GL(2)$ action, the coordinates
of the irreducible representations of $T$ will change, so each
irreducible representation that appears in $\tau^{0}$ must vanish identically.
Hence, $2$-integrability of $M$ by integral manifolds implies
\begin{equation}T_{10}=T_{6}=T_{2}=0.\end{equation}
The remaining torsion components, $\tau^{-2}$ and $\tau^{2}$, are
absorbed by setting $\pi_1 =  2\varphi_2 - 322560 T_{8,-2}\ \omega^4$ and $\pi_2
= -2\varphi_{-2} -322560 T_{8,2}\ \omega^{-4}$, so
\begin{equation}
\mathrm{d}\begin{pmatrix}
\omega^{-2} \\ \omega^{0} \\ \omega^{2}
\end{pmatrix} \equiv
\begin{pmatrix}
\pi_1 & 0 \\
0     & 0 \\
0     & \pi_2  \\
\end{pmatrix}
\wedge
\begin{pmatrix}
\omega^{-4}\\
\omega^{ 4} 
\end{pmatrix} 
, \mod \omega^{-2}, \omega^{0}, \omega^{2}.
\label{absorbed_involutive_structure}
\end{equation}
This proves the torsion condition in the theorem.

Conversely, to establish the existence and parametrization of bi-secant surfaces,
one can apply Cartan's test for involutivity to the tableau in
Equation~(\ref{absorbed_involutive_structure}) \cite{Bryant1991, Ivey2003}.
For a generic flag of $\mathbf{T}_pN$ obtained from generic linear
combinations of $\omega^{-4}$ and $\omega^{ 4}$, the tableau has Cartan
characters $s_1=2$ and $s_2=0$.  Let $\tilde{\nu}:\tilde{N} \to \mathcal{B}$
denote the embedding of the integral surface.  The space of integral elements for the EDS
$(\mathcal{I},\omega^{-4}\wedge\omega^{4})$ is 2-dimensional, as parametrized
by the coefficients $p_{1,4}$ and $p_{3,-4}$ that appear in the pulled-back
forms $\tilde{\nu}^*(\pi_1) = p_{1,4}\tilde{\nu}^*(\omega^4)$ and
$\tilde{\nu}^*(\pi_2)=p_{3,-4} \tilde{\nu}^*(\omega^{-4})$.  Therefore, if
${\pi:\mathcal{B}\to M}$ is 2-integrable, then bi-secant surfaces in $M$
depend on two functions of one variable.  With the Cartan characters computed,
Cartan's test for involutivity applies, so integral surfaces for the linear
Pfaffian system exist in the real-analytic category.  

More can be said by employing modern theorems regarding hyperbolic exterior
differential systems \cite{Yang1987}.  The
characteristic variety of the tableau consists of two real points, and because
a generic tableau is involutive with $s_1=2$, Yang's generalization of the
Cartan--K\"ahler theorem to smooth hyperbolic systems implies that integral
surfaces exist and are parametrized by two functions of one variable in the
smooth category \cite[Theorem~1.19]{Yang1987}.  A special case of this
observation is revisited in Section~\ref{section-planar}.
\end{proof}

By restricting the torsion, Theorem~\ref{bisecant} provides the necessary and
sufficient first-order conditions for 2-integrability.  These first-order
conditions imply syzygies on the second-order invariants via the Bianchi
identity.

\begin{cor}
\label{int2-n4}
Suppose the $GL(2)$-structure ${\pi:\mathcal{B}\to M}$ is 2-integrable with
torsion $T$.  Let $S = \nabla(T)$ denote the covariant derivative of $T$, and
let $Q =T \circ T$ denote the symmetric product of $T$.
Then $\mathcal{B}$ has structure equations of the form
\begin{equation}
\begin{split}
\mathrm{d}\omega &= -\pair{\varphi,\omega}_1 - \pair{\lambda,\omega}_0 +
\pair{T,\pair{\omega,\omega}_1}_5 \\
\mathrm{d}\varphi &= - \frac12\pair{\varphi,\varphi}_1 
+\pair{R^2_0,\pair{\omega,\omega}_3}_0 
+\pair{48 Q_4 + 42 S_4, \pair{\omega,\omega}_3}_2\\
&\phantom{=} +\pair{45 S_6, \pair{\omega,\omega}_1}_5
+\pair{33 S_8, \pair{\omega,\omega}_1}_6
+ \pair{-8Q_4 - 12 S_4, \pair{\omega,\omega}_1}_4\\
\mathrm{d}\lambda &= 960 \pair{S_6, \pair{\omega,\omega}_1}_6\\
\end{split}
\label{int2-n4-str}
\end{equation}
for a scalar curvature function $R^2_0:\mathcal{B} \to \mathcal{V}_0$.
\end{cor}

\begin{proof}
Write $\nabla =
\mathrm{d}+\theta$ for the covariant derivative on $\mathcal{B}$ defined by
the connection 1-form $(\varphi,\lambda)$.  Second-order consequences of
2-integrability arise from the Bianchi identity, $\nabla(\theta)
\wedge\omega = \nabla(T(\omega\wedge\omega))$.  Curvature appears in
$\nabla(\theta)$, which splits into 
$R(\omega\wedge\omega) = \mathrm{d}\varphi +
\frac12\pair{\varphi,\varphi}_1$ and
$r(\omega\wedge\omega) = \mathrm{d}\lambda$.  
The covariant derivative of the torsion two-form,
$\nabla(T(\omega\wedge\omega))$, expands as  $\nabla
(T(\omega\wedge\omega)) = \nabla(T)(\omega\wedge\omega) + 2
Q(T,T)(\omega\wedge\omega\wedge\omega)$, so the Bianchi identity for a
$GL(2)$-structure is
\begin{equation}
R(\omega\wedge\omega)\wedge\omega + r(\omega\wedge\omega)\wedge\omega =
\nabla(T)(\omega\wedge\omega) + 2
Q(T,T)(\omega\wedge\omega\wedge\omega).
\label{bianchi}
\end{equation}
Each of $R$, $r$, $\nabla(T)$ and $Q$ is a function on $\mathcal{B}$ to the
appropriate $SL(2)$ module:
\begin{equation}
\begin{split}
R&:\mathcal{B} \to \mathfrak{sl}(2)\otimes (\mathcal{V}_4^* \wedge
\mathcal{V}_4^*),\\ 
r&:\mathcal{B} \to \mathbb{R} \otimes (\mathcal{V}_4^*\wedge\mathcal{V}_4^*),\\
\nabla(T)&:  \mathcal{B}\to  H^{0,2}(\mathfrak{gl}(2))\otimes
\mathcal{V}_4^*,\ \text{and}\\
Q&: \mathcal{B} \to \mathrm{Sym}^2(H^{0,2}(\mathfrak{gl}(2))\cap
(\mathcal{V}_4 \otimes \wedge^3\mathcal{V}_4^*).
\end{split}
\label{curvtures}
\end{equation}

Since $\mathcal{B}$ is 2-integrable, $T$ takes values only in $\mathcal{V}_8
\subset H^{0,2}(\mathfrak{gl}(2))$.
Using the Clebsch--Gordan decomposition, the irreducible components of these
functions are (omitting the domain $\mathcal{B}$ for brevity):
\begin{equation}\begin{split}
R &= (R^2_0 + R^2_2 + R^2_4) + (R^6_4 + R^6_6 + R^6_8) \in \mathcal{V}_2 \otimes
(\mathcal{V}_2 \oplus \mathcal{V}_6),\\
r &= r_2 + r_6 \in \mathcal{V}_0 \otimes (\mathcal{V}_2 \oplus \mathcal{V}_6),\\
\nabla T &=  S_4 + S_6 + S_8 + S_{10} + S_{12} \in \mathcal{V}_8 \otimes \mathcal{V}_4,\ \text{and}\\
Q &= Q_4 + Q_8 \in \mathrm{Sym}^2(\mathcal{V}_8) \cap (\mathcal{V}_4 \otimes
(\mathcal{V}_2 \oplus \mathcal{V}_6)).
\end{split}\label{bianchicomps}
\end{equation}

Thus, Equation~(\ref{bianchi}) and Schur's lemma together imply linear
relations among the irreducible components listed in Equation~(\ref{bianchicomps}).
To find these relations, one can expand Equation~(\ref{bianchi}) using the
Clebsch--Gordan pairing; for example, one of the terms is
\begin{equation}
\pair{\nabla T, \pair{\omega,\omega}_1}_5
=
\pair{ 
    \pair{S_{ 4}, \omega}_0 
   +\pair{S_{ 6}, \omega}_1 
   +\pair{S_{ 8}, \omega}_2 
   +\pair{S_{10}, \omega}_3 
   +\pair{S_{12}, \omega}_4,
\pair{\omega,\omega}_1}_5.
\end{equation}
The result is 
$S_{10}=0$, 
$R^6_8 = 33 S_8$,
$R^6_6 = 45 S_6$,
$r_6 = 960 S_6$,
$R^6_4 = -12 S_4 - 8 Q_4$, 
$R^2_4 = 42 S_4 + 48 Q_4$,
$r_2 = 0$, and $R^2_2 = 0$.
In particular, $R^2_0$ is the only irreducible component of $r$ and $R$ that
is not an algebraic function of $T$ and $\nabla(T)$.
\end{proof}

\section{3-Integrability}
\label{3int}

\begin{thm} 
If an analytic $GL(2)$-structure ${\pi:\mathcal{B}\to M}$ is 3-integrable, then tri-secant 
3-folds in $M$ are locally parametrized by three functions of one variable.
\label{thm:3param} 
\end{thm}

For 3-integrable $GL(2)$-structures that arise from PDEs of hydrodynamic type
as in Theorem~\ref{thm:F23int} (and are therefore also 2-integrable), this
parametrization by three functions of one variable confirms the computation
presented in \cite{Ferapontov2009}.

\begin{proof}
This theorem is proven by applying Cartan--K\"ahler theory to a differential
ideal whose integral manifolds are tri-secant 3-folds $N \subset M$ through an
arbitrary tri-secant element $E \in Gr_3(\mathbf{T}_pM)$.

As in the proof of Theorem~\ref{bisecant}, consider a local $GL(2)$ coframe
$u$, and let $b=u(p)$.  Since $E$ is tri-secant, $b(E)$ is spanned by $(g_1x
- h_1y)^4$, $(g_2x-h_2y)^4$, and $(g_3x - h_3 y)^4$, distinct, but $u$ can be
adapted so that $b(E)$ is spanned by $x^4$, $y^4$, and $(x+y)^4$.
Therefore
\begin{equation}
b(E) = \{ (A+B)x^4 + B (4x^3y  + 6x^2y^2 + 4xy^3) + (B+C)y^3~:~ A,B,C \in
\mathbb{R}\} \subset \mathcal{V}_4.
\end{equation}
Just as in Theorem~\ref{bisecant}, lift to $\tilde{E} = u_*(E) \in
Gr_2(\mathbf{T}_b \mathcal{B})$, which is
integral to the linear Pfaffian system $\mathcal{I}$ generated by $\kappa^{-2} =
\omega^{-2} - \omega^{0}$ and $\kappa^{ 2} = \omega^{2}-\omega^{0}$ with the
independence condition $\omega^{-4}\wedge\omega^{0}\wedge\omega^{4} \neq 0$.
Again, the projection to $M$ of any integral 3-fold of $\mathcal{I}$ will be a
tri-secant 3-fold that passes through $E$, and every tri-secant 3-fold arises
this way (up to a $GL(2)$ frame adaptation).

The tableau and torsion of $\mathcal{I}$ are given by 
\begin{equation}
\mathrm{d}
\begin{pmatrix}
\kappa^{-2} \\ \kappa^{2} 
\end{pmatrix}\equiv
\begin{pmatrix}
\pi_1 & \pi_3 & 0 \\
0 & -\pi_1 - \pi_2 -\pi_3 & \pi_2
\end{pmatrix}
\wedge
\begin{pmatrix}
\omega^{-4} \\ \omega^{0} \\ \omega^{4}
\end{pmatrix}
+ 
\sum_{a<b}
\begin{pmatrix}
\tau^{-2}_{a,b} \\
\tau^{ 2}_{a,b}
\end{pmatrix}
 \omega^{a}\wedge\omega^{b}
\end{equation}
modulo $\kappa^{-2}, \kappa^{2}$, where $\pi_1 = -2\varphi_{2}$, $\pi_2 =
2\varphi_{-2}$, and $\pi_3 = 2\varphi_{-2} - 4\varphi_{0} + 4\varphi_{2}$.
The apparent torsion can be fully absorbed by redefining $\hat \pi_i = \pi_i -
p_{i,a}\omega^a$ for the parameters
\begin{equation}
\begin{split}
p_{1, 4} &= - \tau^{-2}_{-4,4}\\
p_{2,-4} &= \tau^2_{-4,4} \\
p_{3, 4} &= -\tau^{-2}_{0,4} \\
p_{3,-4} &= - p_{1,-4} - \tau^2_{-4,4} - \tau^2_{-4,0}\\
p_{2, 4} &= -p_{2,0} -p_{3,4} -p_{1,4} + \tau^2_{0,4}  
         = -p_{2,0} +\tau^{-2}_{0,4} +\tau^{-2}_{-4,4} + \tau^2_{0,4}\\
p_{1, 0} &= p_{3,-4} - \tau^{-2}_{-4,4} = -p_{1,-4} -\tau^2_{-4,4} - \tau^2_{-4,0} - \tau^2_{-4,4}.
\end{split}
\end{equation}
The integral elements are still free up to arbitrary choice of three
functions, $p_{1,-4}$, $p_{2,0}$, and $p_{3,0}$.  Since the Cartan characters
are $s_1=2$, $s_2=1$, and $s_3 =0$, but  $s_1+2s_2 + 3s_3 = 4 \neq 3$, the
tableau is not involutive; prolongation is required.

Let $\mathcal{I}^{(1)}$ be the prolonged ideal, which is differentially
generated by the forms $\kappa^{-2}$ and $\kappa^{2}$ along with
\begin{equation}
\begin{split}
\eta^1 &= \pi_1 +p_{1,-4}\ \omega^{-4}-(p_{1,-4} + \tau^2_{-4,4} + \tau^2_{-4,0} + \tau^2_{-4,4})\ \omega^0
- \tau^{-2}_{-4,4}\ \omega^{4},\\
\eta^2 &= \pi_2 + \tau^2_{-4,4}\ \omega^{-4} + p_{2,0}\ \omega^0 + 
(-p_{2,0} +\tau^{-2}_{0,4} +\tau^{-2}_{-4,4} + \tau^2_{0,4})\ \omega^4,
\text{ and }\\
\eta^3 &= \pi_3 + (- p_{1,-4} - \tau^2_{-4,4} - \tau^2_{-4,0})\ \omega^{-4} +
p_{3,0}\ \omega^0 + -\tau^{-2}_{0,4}\ \omega^4.
\end{split}
\end{equation}
After this prolongation, the tableau and torsion are given by 
\begin{equation}
\mathrm{d}\begin{pmatrix}
\kappa^{-2}\\ \kappa^{2}\\
\eta^1 \\ \eta^2 \\ \eta^3 
\end{pmatrix}\equiv
\begin{pmatrix}
0 & 0 & 0 \\
0 & 0 & 0 \\
\pi_4 & -\pi_4 & 0 \\
0 & \pi_5 & -\pi_5 \\
-\pi_4 & \pi_6 & 0 
\end{pmatrix}
\wedge
\begin{pmatrix}
\omega^{-4}\\ \omega^{0}\\ \omega^{4}
\end{pmatrix}
+ \tau^{(1)}(\eta\wedge\eta),
\label{full3int}
\end{equation}
modulo $\kappa^{-2}, \kappa^{2}, \eta^1, \eta^2, \eta^3$.  This tableau has
Cartan characters $s_1=3$, $s_2=0$, and $s_3=0$.  Applying Cartan's test, $s_1
+ 2s_2+ 3s_3  = 3$, which matches the dimension of the variety of
three-dimensional integral elements of $\mathcal{I}^{(1)}$, so the tableau is
involutive.  When the apparent torsion $\tau^{(1)}$ vanishes or can be
absorbed, then in the analytic category Cartan's test for involutivity
implies that the integral 3-folds locally depend on three functions of one
variable.

In fact, one can say more by analyzing the characteristic variety of the involutive
linear Pfaffian system presented in Equation~(\ref{full3int}).   Given the
highest non-zero Cartan character, $s_1=3$, it is clear that the complex
characteristic variety has dimension zero and degree three
\cite[Chapter~V]{Bryant1991}.  It is easy to verify that the characteristic
variety consists of three real points, so each 3-fold integral
to $\mathcal{I}^{(1)}$ is foliated by three families of 2-folds integral to
$\mathcal{I}^{(1)}$.  A special case of this fact is revisited in
Corollary~\ref{holonomic-net}.  
\end{proof}

To determine sufficient conditions for \emph{existence} of tri-secant 3-folds,
one must study the unabsorbable portion of the remaining torsion, $\tau^{(1)}$.
Because $\tau^{(1)}$ is the torsion of the prolonged system
$\mathcal{I}^{(1)}$, it will involve second-order invariants of the
$GL(2)$-structure $\mathcal{B}$ that appear in the Bianchi identity for $\mathcal{B}$,
Equation~(\ref{bianchi}).  Since $T$ is \emph{a priori} valued in
$H^{0,2}(\mathfrak{gl}(2))= \mathcal{V}_2 \oplus \mathcal{V}_6 \oplus
\mathcal{V}_8 \oplus \mathcal{V}_{10}$, components of any of the following
functions may occur in $\tau^{(1)}$:
\begin{equation}
\begin{split}
R &: \mathcal{B}\to \mathcal{V}_2 \otimes  (\wedge^2 \mathcal{V}_4), \\
r &: \mathcal{B}\to  \mathcal{V}_0 \otimes (\wedge^2\mathcal{V}_4), \\
\nabla T &: \mathcal{B}\to (\mathcal{V}_2 \oplus \mathcal{V}_{6} \oplus \mathcal{V}_8 \oplus
\mathcal{V}_{10}) \otimes \mathcal{V}_4, \\
Q &: \mathcal{B}\to \mathrm{Sym}^2(\mathcal{V}_2 \oplus \mathcal{V}_6 \oplus
\mathcal{V}_{8} \oplus \mathcal{V}_{10}) \cap ( \mathcal{V}_4 \otimes \wedge^3
\mathcal{V}_4 ).
\end{split}
\label{eqn:int3-curv}
\end{equation}
The vanishing of the unabsorbable portion of $\tau^{(1)}$ will place
restrictions on the various irreducible representations appearing in
Equation~(\ref{eqn:int3-curv}).  The enormous complexity of $Q$ and $\nabla T$ makes
decomposition of the unabsorbable portion of $\tau^{(1)}$ extremely difficult.
Fortunately, one can make a simplifying assumption that is consistent with the
motivating PDE theory in Theorem~\ref{thm:F23int}; henceforth, all theorems
discuss only those $GL(2)$-structures that are both 2-integrable and
3-integrable.

\begin{cor} 
If $\mathcal{B}\to M$ is a 2,3-integrable $GL(2)$-structure, then any tri-secant $N^3
\subset M$ is triply foliated by bi-secant surfaces.  Moreover, the net
defined by $\mathbf{T}N \cap \mathbf{C}$ is a coordinate net. 
\label{holonomic-net}
\end{cor}

The triple foliation by bi-secant surfaces appears to be new, though it is not
surprising based on the description of highest-weight polynomial subspaces of
$\mathcal{V}_n$ seen in \cite{Crooks2009}.  This foliation shows that the
``holonomic characteristic net'' condition for tri-secant 3-folds seen in
\cite{Ferapontov2009} is superfluous since 2-integrability is implied for
$GL(2)$-structures arising from Hessian hydrodynamic PDEs.

\begin{proof}
Without loss of generality in a contractible neighborhood in $M$, one may
apply a $GL(2)$ change-of-frame so that  
$\mathbf{T}N = \{ Ax^4 + B(x+y)^4 +Cy^4\}$. 
The characteristic net, $\mathbf{T}N \cap \mathbf{C}$, is given by its tangent
planes 
\begin{equation}
\begin{split}
K_1 &= \mathop{\mathrm{span}}\{ (x+y)^4, y^4\},\\
K_2 &= \mathop{\mathrm{span}}\{ x^4, y^4\},\ \text{and}\\
K_3 &= \mathop{\mathrm{span}}\{ x^4, (x+y)^4\}.
\end{split}
\end{equation}
Each of $K_1$, $K_2$, and $K_3$ is clearly bi-secant.  It suffices to prove
that they are everywhere tangent to the level sets of a coordinate system on
$N$, as these level sets will be bi-secant surfaces.

Recall that $N$ is defined by the projection of an integral manifold of  
the prolonged system $\mathcal{I}^{(1)}$ on
$\mathcal{B}\times \mathbb{R}^3$ from Theorem~\ref{thm:3param} .
In particular, there is a submanifold $\tilde{N} \subset \mathcal{B}$ with
embedding map $\tilde{\nu}:\tilde{N} \to \mathcal{B}$ depending on three
parameters $P_1, P_2, P_3$ such that each $\mathbf{T}_p N$ is the projection of 
\begin{equation}
\begin{split}
\mathbf{T}_b\tilde{N} &= \ker \big(\{ \kappa^{-2}, \kappa^{2}, \eta^1, \eta^2,
\eta^3,\\ 
&\phantom{=\ker( \{ }
\pi_4 - P_1 (\omega^{-4}  - \omega^{0}),\\
&\phantom{=\ker( \{ }
\pi_5 - P_2 (\omega^{0}  - \omega^{4}),\\
&\phantom{=\ker( \{ }
\pi_6 - P_1 \omega^{-4} +  P_3 \omega^{0}\}\big).
\end{split}
\end{equation}

Let $\Upsilon^1 = \omega^{-4}-\omega^{0}$, $\Upsilon^2 = \omega^0$, and
$\Upsilon^3=\omega^4 - \omega^0$.  Therefore, at each basepoint, $K_i \subset
\mathbf{T}_p N$ is the projection of the plane $\ker(\tilde{\nu}^* (\Upsilon^i))
\subset \mathbf{T}_b \tilde{N}$.  One can easily compute that
$\mathrm{d}(\tilde{\nu}^*(\Upsilon^i)) \equiv 0$ modulo $\tilde{\nu}^*(\Upsilon^i)$ for each
$i$; therefore, the Frobenius theorem provides local
coordinates $(s^1, s^2, s^3)$ on $N$ such that $K_i = \ker(\mathrm{d}s^i)$.
\end{proof}
When $\mathcal{B}$ arises from a Hessian hydrodynamic PDE, these coordinates
are called the Riemann invariants of the hydrodynamic reduction.

\begin{thm}[2,3-integrable $GL(2)$-structure equations] 
A $GL(2)$-structure $\mathcal{B}$ is 2,3-integrable if and only if the torsion
$T$ of $\mathcal{B}$ only takes values in $\mathcal{V}_8$ and the curvature is a
function of $T$.  In particular, every 2,3-integrable $GL(2)$-structure $\mathcal{B}$ has the
following structure equations
\begin{equation}
\begin{split}
\mathrm{d}\omega &= 
-\pair{\varphi,\omega}_1 
-\pair{\lambda,\omega}_0 +
 \pair{T, \pair{\omega,\omega}_1}_5\\
\mathrm{d}\lambda &= 0\\
\mathrm{d}\varphi &= -\frac12\pair{\varphi,\varphi}_1 
 -2080        \pair{\pair{T,T}_8,\pair{\omega,\omega}_3}_0 
 +64          \pair{\pair{T,T}_6,\pair{\omega,\omega}_3}_2 \\
&\quad-\frac{88}{7}\pair{\pair{T,T}_6,\pair{\omega,\omega}_1}_4  
 +\frac{24}{7}\pair{\pair{T,T}_4,\pair{\omega,\omega}_1}_6\\
\mathrm{d}T &= J(T) \begin{pmatrix}\omega\\\lambda\\\varphi\end{pmatrix}
\end{split}
\label{eqn:int3-n4-str}
\end{equation}
for a $9 \times 9$ matrix-valued function $J$ whose entries are linear and quadratic
polynomials in the coefficients of $T$, as provided in Appendix~\ref{appxJ}
and \cite{GL2maplecode, Smith2009}.
\label{int3-n4-str}
\end{thm}

\begin{proof}
2-integrability implies that $T=T_8$, so $R$, $r$, $\nabla(T)$, and $Q$
decompose as in Equation~(\ref{bianchicomps}) with the relations implied in
Theorem~\ref{bisecant}.  3-integrability implies the vanishing of the $GL(2)$
orbit of the unabsorbable portion of $\tau^{(1)}$ in
Equation~(\ref{full3int}).  The vanishing of $\tau^{(1)}$ and the Bianchi
identity ($\mathrm{d}^2 = 0$) together imply the additional relations
$S_{12}=0$, $7 S_8= 24 Q_8 = 24 \pair{T,T}_4$, $S_6 = 0$, $21S_4 = 8 Q_4 = 8
\pair{T,T}_6$, and $R^2_0=-2080 \pair{T,T}_8$.

Therefore, the curvatures, $R$ and $r$, are quadratic functions of $T$.  The
derivative of torsion, $\mathrm{d}T$, is also a quadratic function of $T$, as
expressed in the matrix $J(T)$. Hence, $T$ is the only invariant of any order
for 2,3-integrable $GL(2)$-structures.
\end{proof}

\begin{defn} 
The notation $(\mathcal{B},M,p)_{2,3}$ indicates a smooth 2,3-integrable $GL(2)$-structure
$\pi:\mathcal{B}\to M$ such that $p \in M$ and such that $M$ is connected.  
\end{defn}

The condition that $M$ is connected is crucial in what follows, and
pointedness is technically useful.

\begin{defn}[Representatives]
$(\mathcal{B},M,p)_{2,3}$ \textbf{represents} $v \in \mathcal{V}_8$ if $v \in T(\mathcal{B}_p)$.
More generally, a 2,3-integrable $GL(2)$-structure $\mathcal{B}\to M$ \textbf{represents}
$v \in \mathcal{V}_8$ if $T(\mathcal{B})\owns v$.  Likewise, a Hessian hydrodynamic PDE
\textbf{represents} $v \in \mathcal{V}_8$ if the induced 2,3-integrable
$GL(2)$-structure over $M=F^{-1}(0)$ represents $v$.  
\end{defn}

\begin{lemma} 
The singular distribution on $\mathcal{V}_8$ defined by the
columns of the matrix $J$ is integrable, providing a stratification of
$\mathcal{V}_8$ into leaves that are submanifolds.   That is, for any $v \in
\mathcal{V}_8$, there exists a unique submanifold $\mathcal{O}_J(v)$ such that
$\mathbf{T}_v \mathcal{O}_J(v) = \mathop{\mathrm{range}}J(v)$.  
\label{lemmaT}
\end{lemma}

Lemma~\ref{lemmaT} is elementary in its modern interpretation using the theory of Lie algebroids and
smooth groupoids.  Since the equations in Theorem~\ref{int3-n4-str} are closed
under exterior derivative, the matrix $J$ is closed under the corresponding bracket, so it defines the anchor map
of a Lie algebroid over $\mathcal{V}_8$.  This Lie algebroid is neither
regular nor transitive, but there exists an integrating groupoid over the
base, $\mathcal{V}_8$.  This groupoid is smooth, but it may not be a Lie
groupoid. However, the groupoid is transitive when restricted to each of its
orbits, $\mathcal{O}_J(v)$, and each orbit is a submanifold of the base
\cite{Crainic2003, Mackenzie2005, Stefan1980}.  These orbits may be regarded
as the leaves of a singular foliation of $\mathcal{V}_8$ such that
$\mathbf{T}_v \mathcal{O}_J(v)$ is spanned by the columns of $J(v)$.  

\begin{cor}
For any $v \in \mathcal{V}_8$, there exists a real-analytic connected
2,3-integrable $GL(2)$-structure $\mathcal{B}\to M$ such that $T(\mathcal{B})
\owns v$.  That is, every $v \in \mathcal{V}_8$ is represented by some
real-analytic $(\mathcal{B},M,p)_{2,3}$.  Moreover, $J$ has constant rank on any
such $\mathcal{B}$.
\label{T_exist}
\end{cor}

\begin{proof} Since the structure equations in Theorem~\ref{int3-n4-str} are
closed under exterior derivative, this is a direct application of the
existence part of Cartan's structure theorem, which is an generalization of
Lie's third fundamental theorem to the intransitive case
\cite{Cartan1904}.   See Appendix A of
\cite{Bryant2001a} for a clear summary of the special case of Cartan's
structure theorem needed here.  As used here, Cartan's structure
theorem holds in the smooth category; however, because $T \circ T$ and $J(T)$
are real-analytic (in fact, linear and quadratic), one can 
use the Cartan--K\"ahler machinery to produce a local real-analytic solution 
$\mathcal{B}$ of the structure equations.
On any solution $\mathcal{B}$, the image of $\mathrm{d}T|_b$ is the image of
the columns of the matrix $J(T(b))$, and these columns span the tangent space
of the submanifold $\mathcal{O}_J(T(b))$.  Thus, on any connected
$\mathcal{B}$ containing $b$, the rank of $J$ equals the dimension of
$\mathcal{O}_J(T(b))$.
\end{proof}

\begin{cor}  
$(\mathcal{B},M,p)_{2,3}$ and $(\hat{\mathcal{B}}, \hat{M}, \hat{p})_{2,3}$ admit a local
$GL(2)$-equivalence $f:M \to \hat{M}$ with $f(p)=\hat{p}$ if and only if
$T(\mathcal{B}_p) \cap \hat{T}(\hat{\mathcal{B}}_{\hat{p}}) \neq \emptyset$.  That
is, the value $T(b)$ uniquely defines $(\mathcal{B},M,p)_{2,3}$ in a
neighborhood of $p=\pi(b)$, and this local $GL(2)$-structure is real-analytic.
\label{T_unique}
\end{cor}

\begin{proof} 
This is a direct application of the local uniqueness part of Cartan's
structure theorem.  Again, the theorem applies here in the smooth category,
but the real-analyticity of the structure equations implies that any
2,3-integrable $GL(2)$-structure is locally equivalent to the real-analytic
$GL(2)$-structure produced in Corollary~\ref{T_exist}.
\end{proof}

Corollaries~\ref{T_exist} and \ref{T_unique} establish the primacy of $T$ in
the study 2,3-integrable $GL(2)$-structures and the related PDEs.  However,
Corollary~\ref{T_unique} provides only local GL(2)-equivalence, so it is
useful to have a weaker ``chain-wise'' notion of equivalence that applies to
global (but still connected) $GL(2)$-structures.  

\begin{defn}[Leaf-equivalence]
$(\mathcal{B}_0,M_0,p_0)_{2,3}$ and $(\mathcal{B}_k,M_k,p_k)_{2,3}$ are said to be
\textbf{leaf-equivalent} if there exist finite sequences
$\{(\mathcal{B}_i,M_i,p_i)_{2,3}\}$ and $\{v_i\}$
with $1 \leq i \leq k{-}1$ such that $(\mathcal{B}_i,M_i,p_i)_{2,3}$ and
$(\mathcal{B}_{i+1},M_{i+1},p_{i+1})_{2,3}$ both represent $v_i$ for $0 \leq i
\leq k{-}1$. 
\end{defn}

The term ``leaf-equivalence'' arises from the leaves of the singular foliation
of $\mathcal{V}_8$ from Lemma~\ref{lemmaT}.  These leaves separate all
possible $(\mathcal{B},M,p)_{2,3}$'s into equivalence classes by the value of
$T$.  
By slight abuse of notation, if $(\mathcal{B},M,p)_{2,3}$ represents $v$, then write
$\mathcal{O}_J(\mathcal{B})$ to denote $\mathcal{O}_J(v)$.  Leaf-equivalence
can now be rewritten more succinctly.

\begin{thm}[Leaf-equivalence]
$(\mathcal{B},M,p)_{2,3}$ and $(\hat{\mathcal{B}}, \hat M, \hat p)_{2,3}$ are
leaf-equivalent if and only if $\mathcal{O}_J(\mathcal{B}) =
\mathcal{O}_J(\hat{\mathcal{B}})$.  Moreover, $T:\mathcal{B} \to
\mathcal{O}_J(\mathcal{B})$ is a submersion.  \label{leaves} \end{thm}
Although $T:\mathcal{B} \to \mathcal{O}_J(\mathcal{B})$ is an open map, it
need not be surjective.

\section{The Classification}
\label{moduli}
To classify connected 2,3-integrable $GL(2)$-structures is to explicitly
identify the leaves of the foliation of $\mathcal{V}_8$.   To identify the
leaves of the foliation is to identify connected $GL(2)$-invariant
submanifolds of $\mathcal{V}_8$ where the rank of $J(v)$ is constant. This is
all ultimately achieved in Theorem~\ref{thm:classes} thanks to the
observations in Lemmas~\ref{matrixmiracles} and \ref{root types}.

\begin{lemma}
In vector space of 14th degree polynomials in the variables $v_{-8}, \ldots,
v_{8}$, the determinant of $J(v)$ is a non-zero scalar multiple of the discriminant of the
polynomial $v\in \mathcal{V}_8$.  Therefore, $J(v)$ is non-singular if and
only if $v$ has eight distinct roots.
Moreover, if $v$ is a nontrivial polynomial with $k$ distinct roots, then the
rank of $J(v)$ is $k{+}1$.  
\label{matrixmiracles}
\end{lemma}

\begin{proof} 
The determinant statement is verified by computing the g.c.d.\ of
the two polynomials, $\det J(v)$ and $\mathop{\mathrm{disc}}(v)$.
The rank statement is directly verified by writing
$v=(g_1x-h_1y)(g_2x-h_2y)\cdots(g_8x-h_8y)$, imposing multiplicity on the
$h_i$'s and $g_i$'s and computing the rank of $J(v)$ directly.
\end{proof}

As it happens, a simple exercise in linear fractional transformations shows
that the multiplicity of the roots is preserved by the $GL(2)$ action on
$\mathcal{V}_8$ \cite{Ahlfors1978, Conway1978}.

\begin{lemma}
For $v(x,y) \in \mathcal{V}_8$, the multiplicity and complex type of the roots
are preserved under the irreducible action of $GL(2,\mathbb{R})$.
\label{root types}
\end{lemma}

%
%
%

Bearing in mind this lemma, notational shorthand for the \textbf{root type} of
a polynomial $v(x,y) \in \mathcal{V}_8$ is useful:  Suppose $v(x,y)$ factors
as $(g_1 x - h_1 y)^{r_1}\cdots(g_m x - h_m y)^{r_m}$ such that $r_1 \geq r_2
\geq \cdots \geq r_m$ and $r_1+r_2 + \cdots + r_m = 8$.  These exponents
define a partition of $8$ that is written as $\{r_1, r_2, \ldots, r_m\}$.  If
the roots  $g_k/h_k$ and $g_{k+1}/h_{k+1}$ are complex conjugates, then denote
this by enclosing their exponents in square-braces: $\{r_1, \ldots, [r_k,
r_{k+1}], \ldots, r_m\}$.  Denote the root type containing $v(x,y)$ by
$[v(x,y)]$.  A root type is equivalent to the subset of $\mathcal{V}_8$
comprised of all polynomials that factor according to the given partition and
complex-conjugate pairing.  For example, $[x^4(x+iy)^2(x-iy)^2] =
[(x-5y)^4(x+2iy)^2(x-2iy)^2] =\{4,[2,2]\} \subset \mathcal{V}_8$.  Let $\{0\}$
denote the trivial root type, the zero polynomial.  There are 54 non-trivial
root types that partition $\mathcal{V}_8$ into strata from dimension two to
dimension nine, as represented in Figure~\ref{fig:classes-closures}.  

In Figure~\ref{fig:classes-closures}, arrows mean ``closure contains.'' Shaded
nodes represent root types that contain exactly one $GL(2)$ orbit.  Oval nodes
represent open root types. The square node represents the nearly-closed
root type, $\{8\}$, which (when $0$ is included) is the rational normal cone
in $\mathcal{V}_8$.  Hexagonal nodes represent root classes that are neither
closed nor open.  Note that strictly real root types actually have two
connected components (for example, $\{8\}$ is comprised of the two ends of the
rational normal cone), but this is another artifact of the projective nature
of the group action, as the two ends are in the same $GL(2)$ orbit by the
$-I_9$ action. 

\begin{figure}
\begin{centering}
\includegraphics[width=\textwidth]{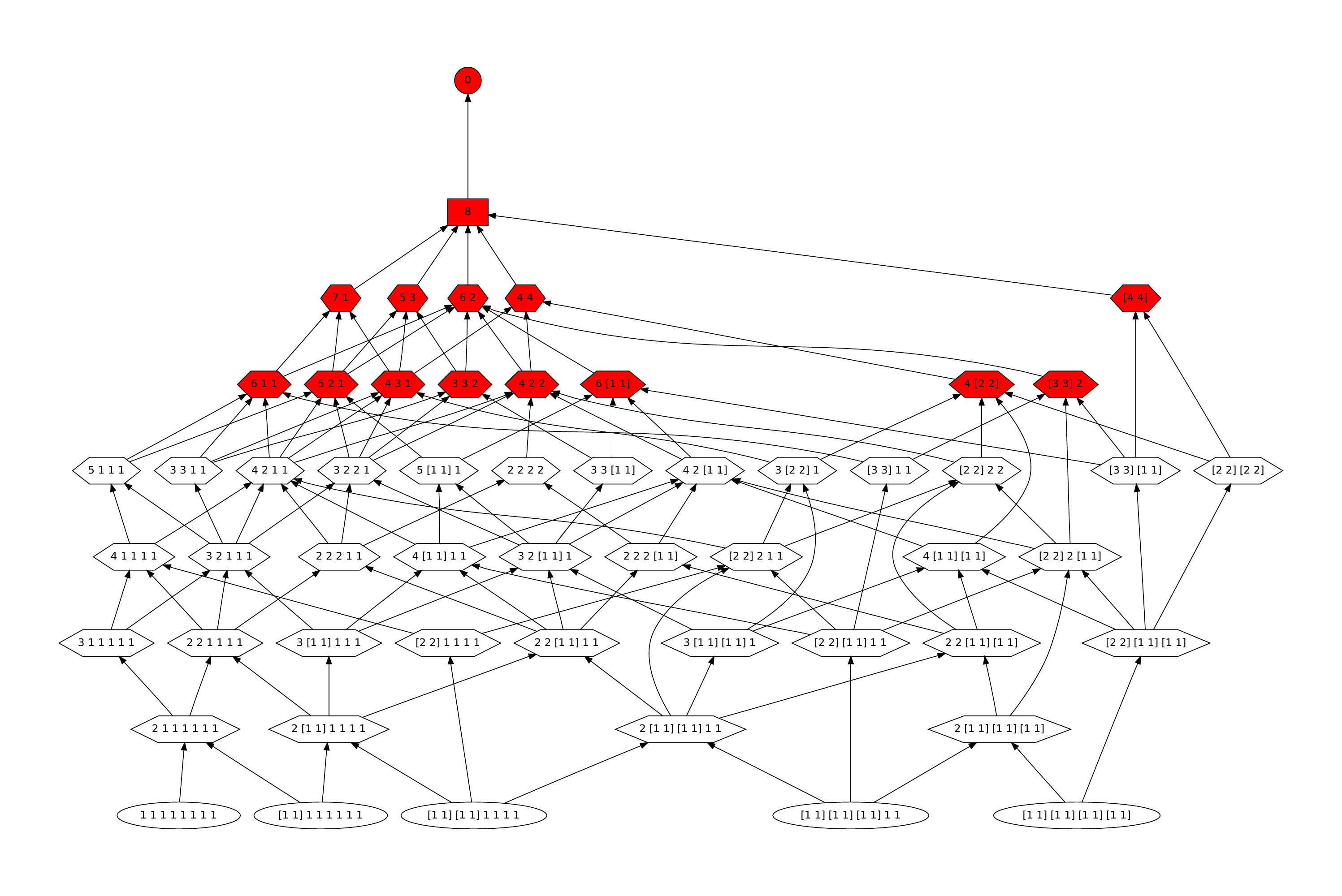}
\caption{The stratification of $\mathcal{V}_8$ into the 55 
root types, sorted by dimension.  Shaded nodes indicate
root types that contain exactly one $GL(2)$ orbit.  The square node is closed (if 0
is included); the oval nodes are open, and the hexagonal nodes are neither
closed nor open.} \label{fig:classes-closures} \end{centering}
\end{figure}

\begin{thm}[Leaf-Equivalence Classes]
\label{thm:classes}
The leaf-equivalence classes of connected 2,3-integrable $GL(2)$-structures are exactly
the root types in $\mathcal{V}_8$.
That is, for all $v \in \mathcal{V}_8$, $\mathcal{O}_{J}(v) = [v]$.
\label{classification}
\end{thm}

\begin{proof}
For each $v$, both $\mathcal{O}_J(v)$ and $[v]$ are smooth submanifolds of
$\mathcal{V}_8$, and by Lemma \ref{matrixmiracles} they have the same
dimension.  To prove that $\mathcal{O}_J(v) = [v]$, it suffices to prove that
$\mathbf{T}_v \mathcal{O}_J(v) = \mathbf{T}_v [v]$ for all $v \in
\mathcal{V}_8$.

Fix a root type $[v]$ and an arbitrary $v \in [v]$.  Note that \[
\mathbf{T}_v\mathcal{O}_{J}(\mathcal{B}) = DT_b(\mathbf{T}_b \mathcal{B}) =
\mathop{\mathrm{range}}J(v).\] Thus it suffices to find, for each column
$J_i(v)$, a tangent vector $D(v) \in \mathbf{T}_v[v]$ such that $D(v) =
J_i(v)$. 
As it happens, the equations defining the various $D(v)$ are easily solvable
at arbitrary points in all 54 non-trivial root types and for all columns of
$J$.  Because both the leaf-equivalence classes and the root types partition
$\mathcal{V}_8$ by smooth submanifolds, and because
$\mathbf{T}_v\mathcal{O}_{J}(v)=\mathbf{T}_v[v]$ for all $v \in
\mathcal{V}_8$, the partitions must be identical.

To illustrate the computations, consider an arbitrary point $v=(g x + h y)^8$
in the root type $[v]=\{8\}$.  An arbitrary element of $\mathbf{T}_v[v]$ looks like 
\begin{equation}
D(v)=8(Gx + Hy)(gx+hy)^7 = 
\begin{pmatrix}
8{ G}{{ g}}^{7}\\ 
{ H}{{ g}}^{7}+7{ G}{ h}{{ g}}^{6}\\ 
2{ H}{ h}{{ g}}^{6}+6{ G}{{ h}}^{2}{{ g}}^{5}\\ 
3{ H}{{ h}}^{2}{{ g}}^{5}+5{ G}{{ g}}^{4}{{ h}}^{3}\\ 
4{ G}{{ g}}^{3}{{ h}}^{4}+4{ H}{{ g}}^{4}{{ h}}^{3}\\ 
5{ H}{{ g}}^{3}{{ h}}^{4}+3{ G}{{ h}}^{5}{{ g}}^{2}\\ 
6{ H}{{ h}}^{5}{{ g}}^{2}+2{ G}{{ h}}^{6}{ g}\\ 
7{H}{{ h}}^{6}{ g}+{ G}{{ h}}^{7}\\ 
8{ H}{{ h}}^{7}
\end{pmatrix}.
\end{equation}  Therefore, one must solve $D(v) = J_i(v)$ for $G$
and $H$ in each of the columns $i=1,\ldots,9$.  In this case, 
\begin{equation}
J( v )
=
\begin{pmatrix}
0&0&0&0&0&{g}^{8}&-16h{g}^{7}&16{g}^{8}&0\\ 
0&0&0&0&0&h{g}^{7}&-14{h}^{2}{g}^{6}&12h{g}^{7}&2{g}^{8}
\\ 
0&0&0&0&0&{h}^{2}{g}^{6}&-12{h}^{3}{g}^{5}&8{h}^{2}{g}^{6}&4h{g}^{7}
\\ 
0&0&0&0&0&{h}^{3}{g}^{5}&-10{h}^{4}{g}^{4}&4{h}^{3}{g}^{5}&6{h}^{2}{g}^{6}
\\ 
0&0&0&0&0&{h}^{4}{g}^{4}&-8{h}^{5}{g}^{3}&0&8{h}^{3}{g}^{5}\\ 
0&0&0&0&0&{h}^{5}{g}^{3}&-6{h}^{6}{g}^{2}&-4{h}^{5}{g}^{3}&10
{h}^{4}{g}^{4}\\ 
0&0&0&0&0&{h}^{6}{g}^{2}&-4{h}^{7}g&-8{h}^{6}{g}^{2}&12{h}^{5}{g_
{{1}}}^{3}\\ 
0&0&0&0&0&{h}^{7}g&-2{h}^{8}&-12{h}^{7}g&14{h}^{6}{g}^{2}\\ 
0&0&0&0&0&{h}^{8}&0&-16{h}^{8}&16{h}^{7}g
\end{pmatrix}
\end{equation}
For each column $J_i(v)$, it is easy to see the solution values of $G$ and $H$.

For the other root types, the computations are similar but somewhat more
complicated.  All that matters is the fact that they can be solved for
arbitrary $v$.  
\end{proof}

\subsection{Symmetry Reduction}
\label{symmred}
Geometrically, it is interesting to reduce all symmetry from a structure.  Fix
$(\mathcal{B},M,p)_{2,3}$ and $v \in T(\mathcal{B})$.  Let $\mathcal{B}^v = \{ b \in
\mathcal{B} : T(b) = v \}$.  This is a sub-bundle of $\mathcal{B}$ with fiber
group $\mathrm{Stab}(v) \subset GL(2)$.  If $\dim[v] \leq 3$, then the
stabilizer group is smooth; however, the stabilizer groups must be discrete
for all the larger root types.   The stabilizer group, either smooth or
discrete, must always appear in the well-known list of
$GL(2,\mathbb{C})$-stabilizers of polynomials \cite{Berchenko2000}.  In
the discrete case, the fiber must therefore be either a cyclic group or
a dihedral group \cite{Smith2009}. 

The structure equations for $\mathcal{B}^v$ show no dependence on $T$ (as it
has been fixed), so a neighborhood in $\mathcal{B}^v$ is a local Lie group of
dimension ${9-\dim[v]}$, and $M$ is locally the homogeneous space
$\mathcal{B}^v/\mathrm{Stab}(v)$.  The relations defined by
$\ker\mathrm{d}T(v)= \ker J(v)$ determine the pull-backs of $\omega$,
$\lambda$, and $\varphi$ to $\mathcal{B}^v$, so one may explicitly reduce
Equation~(\ref{eqn:int3-n4-str}) to obtain structure equations for the local
Lie group $\mathcal{B}^v$. 

If $v=0$, then $\mathrm{Stab}(v) = GL(2)$, so the local flat 2,3-integrable
$GL(2)$-structure is the local Lie group of dimension 9 obtained by setting
$T=0$ in Equation~(\ref{eqn:int3-n4-str}). 

The root type $\{8\}$ is a single $GL(2)$ orbit, so any representative
will generate all representatives. Suppose $v = x^8$.  
Then $\mathrm{d}T_{-8}(x^8) = \lambda + 16 \varphi_0$ and $\mathrm{d}T_{-6}(x^8) =
2 \varphi_2$ while $\mathrm{d}T_{k}=0$ for $k > -6$.  Since $\mathrm{d}T$ is a
vertical 1-form on $\mathcal{B}$, the value of $T$ only varies in the fiber of
$\mathcal{B}$.
The reduced structure equations are 
\begin{equation}
\begin{split}
\mathrm{d} \omega^{-4} 
& = 24\ \varphi_{0}\wedge\omega^{-4}-8\ \varphi_{-2}\wedge\omega^{-2}+
2\cdot322560\ \omega^{0} \wedge \omega^{4},\\
\mathrm{d}\omega^{-2} 
&=  20\ \varphi_{0}\wedge\omega^{-2}-6\ \varphi_{-2}\wedge\omega^{0} + 
322560\ \omega^{2} \wedge \omega^{4},\\
\mathrm{d}\omega^{0} 
&= 16\ \varphi_{0}\wedge\omega^{0} - 4\ \varphi_{-2}\wedge\omega^{2},\\
\mathrm{d}\omega^{2} 
&= 12\ \varphi_{0}\wedge\omega^{2} - 2\ \varphi_{-2}\wedge\omega^{4},\\
\mathrm{d}\omega^{4}&=  8\ \varphi_{0}\wedge\omega^{4},\\
\mathrm{d}\varphi_{0}&=0, \quad
\mathrm{d}\varphi_{-2}= 4\ \varphi_{0}\wedge\varphi_{-2}.
\end{split}
\label{eqn:x8-red}
\end{equation}
These equations can be easily integrated, so $\mathcal{B}^{x^8}$ has coordinates
$\xi^{-4}$, $\xi^{-2}$, $\xi^{0}$, $\xi^{2}$, $\xi^{4}$, $a$, and $b$ such
that
\begin{equation}
\begin{split}
\varphi_{0}&=a^{-1}\ \mathrm{d}a\\
\varphi_{-2}&=a^4\ \mathrm{d}b,\\
\omega^{4}&=a^8\ \mathrm{d}\xi^{4},\\
\omega^{2}&=a^{12}\left(\mathrm{d}\xi^{2}-2b\ \mathrm{d}\xi^{4}\right),\\
\omega^{0}&=a^{16}\left(\mathrm{d}\xi^{0}-4b\ \mathrm{d}\xi^{2} + 4b^2\ \mathrm{d}\xi^{4}\right),\\
\omega^{-2}&=a^{20}\left(\mathrm{d}\xi^{-2}-6b\ \mathrm{d}\xi^{0}+12b^2\
\mathrm{d}\xi^{2}-8 b^3\ \mathrm{d}\xi^{4}-322560\xi^{4}\ \mathrm{d}\xi^{2}\right),\\
\omega^{-4}&=a^{24}\big(\mathrm{d}\xi^{-4}-8b\ \mathrm{d}\xi^{-2}+24b^2\
\mathrm{d}\xi^{0}-32b^3\ \mathrm{d}\xi^{2}+ 16b^4\ \mathrm{d}\xi^{4}+\\ 
&\qquad 8(322560) \xi^{4}b\ \mathrm{d}\xi^{2}-2(322560)\xi^{4}\ \mathrm{d}\xi^{0}\big).
\end{split}
\label{8int}
\end{equation}
The field of rational normal cones can now be written in these local
coordinates, since $\mathbf{C}_p = \{ \pair{u_p(v),u_p(v)}_2~:~ v \in
\mathbf{T}_pM  \}$ for a local section $u$ of $\mathcal{B}$.

Each of the 3-dimensional root types is a single $GL(2)$-orbit.  Again,
$T(\mathcal{B}) = [v]$ for any $v$ in the root type, so an arbitrary
representative $v$ may be chosen for any $(\mathcal{B},M,p)_{2,3}$.  Each $v$ has a
1-dimensional stabilizer, which is a Lie subgroup of $GL(2)$.   The
corresponding Lie algebras are easy to compute \cite{Berchenko2000,
Smith2009}. In all cases, $\mathrm{d}T$ has both vertical and semi-basic
components, so the embedding of the stabilizer fiber group varies over $M$.
The reduced structure $\mathcal{B}^v$ is a $6$-dimensional Lie group.

If $\dim[v]=4$, then $\mathcal{B}^v$ is a finite cover of $M$, but there are
only eight such structures, since each root type with three roots is a single
$GL(2)$-orbit.  This contrasts with the case $\dim[v] \geq 5$.  If $\dim [v]
\geq 5$, then $\mathcal{B}^v$ is a finite cover of $M$, but there is no reason
to believe that $T(\mathcal{B}) \subset [v]$ implies $T(\mathcal{B})=[v]$ when
$\dim [v] \geq 5$.  Consider $[v]=\{2, 2, 2, 2\}$.  The quotient space
$\{2,2,2,2\}/GL(2)$ has orbifold singularities; for example, the point
$x^2y^2(x+y)^2(x-y)^2$ has an eight-element stabilizer group, but all nearby
points have trivial stabilizer groups.  The existence of these orbifold
singularities implies that there cannot be a surjective smooth map $M \to
[v]/GL(2)$.  If $T(\mathcal{B})$ is a proper subset of $[v]$ for all
$\mathcal{B}$ representing $v$, then by Theorem~\ref{classification} a
finite sequence of 2,3-integrable $GL(2)$-structures connects any two points in
the leaf, but infinitely many locally distinct $GL(2)$-structures are required to
cover the entire leaf.  This behavior is closely related to the topology of
$\mathbb{CP}^n$ with $k$ marked points, up to $GL(2,\mathbb{C})$ action, which
is a difficult and well-known problem in complex algebraic geometry.  The real
case encountered here is both harder and less well-known than the classical
complex case.  The most relevant result is \cite{Chu2007}, which shows that the
orbifold singularities present in the larger leaves of $\mathcal{V}_8$ are so
bad that they preclude the existence of Riemannian metrics on the leaves.

In the open case, $\dim[v]=9$, so $J(T)$ has maximum rank, and $T$ itself
provides local coordinates on $\mathcal{B}$.  

This discussion is summarized by the shape of the nodes in
Figure~\ref{fig:classes-equiv}.

\section{GL(2) PDEs}
\label{PDEs}

This section contains various conclusions regarding Hessian hydrodynamic PDEs
that can be inferred from Theorem~\ref{thm:F23int} and
Theorem~\ref{classification}.
Recall that a PDE ${F=0}$ is said to have $k$ symmetries if the Lie algebra of
point symmetries of $F^{-1}(0)$ has dimension $k$. 

\begin{lemma}
A Hessian hydrodynamic PDE $F(u_{ij})=0$ representing $v \in \mathcal{V}_8$
has $k$ symmetries if and only if $\dim [v]=9-k$.
\end{lemma}
\begin{proof}
The contact transformations by $CSp(3)$ induce the local automorphisms on
$\mathcal{B}$, but an automorphism $\psi:\mathcal{B} \to \mathcal{B}$ near
$b$ is a symmetry if and only if the structure equations are preserved by
$\psi^*$.  This can happen if and only if $T(b)$ is preserved near
$\psi(b)$.  In particular,  $\psi^* ( \mathrm{d}T )=\mathrm{d}T$ if and
only if $\psi^*$ is the identity on the range of $\mathrm{d}T$, which has
dimension $\dim[v]$.
\end{proof}

\begin{cor}
There is no Hessian hydrodynamic PDE with
exactly eight symmetries.
\label{PDEnone}
\end{cor}
\begin{proof}
There is no root type of dimension one. 
\end{proof}

The classification provides even more bountiful information about the Hessian
hydrodynamic PDEs, because Theorem~\ref{int3-n4-str} allows a converse of
Theorem~\ref{thm:F23int}.  

\begin{thm}
Every $(\mathcal{B},M,p)_{2,3}$ is realized, locally near $p$, by a Hessian hydrodynamic PDE.
\label{embed}
\end{thm} 

\begin{proof}
The structure $(\mathcal{B},M,p)_{2,3}$ is realized, locally near $p$, by a
Hessian hydrodynamic PDE as in Lemma~\ref{hyperbolics} if and only if there
exists a neighborhood $M' \subset M$ of $p$ and an embedding $i:M' \to
\Lambda^o$ such that the distribution of rational normal cones
$i_*(\mathbf{C}(M'))$ over $i(M')$ is the same as the intersection of
$i_*(\mathbf{T}M')$ with the distribution of Veronese cones over $\Lambda^o$. 

Since the fibers $\mathcal{B}$ are exactly the symmetries of $\mathbf{C}$ and
the fibers of $CSp(3)^o$ are exactly the symmetries of the distribution of
Veronese cones, it suffices to establish a bundle immersion $h:\mathcal{B}(M')
\to CSp(3)^o$ covering an embedding $i:M' \to \Lambda^o$. 

Let $\mu$ denote the Maurer--Cartan form of $CSp(3)$, which is of the form 
\begin{equation}
\mu = 
\begin{pmatrix}
\beta & \gamma \\
\alpha & -\beta^t \\
\end{pmatrix},\ 
\alpha=\alpha^{t},\ \gamma=\gamma^{t},\ \mathrm{d}\mu + \mu\wedge\mu = 0.
\end{equation}
(One may assume that the conformal scaling has been incorporated into $\beta$,
as it is below.)
By Equations~(\ref{Sp3Fiber}) and (\ref{Sp3U}), the components $\beta$ and
$\gamma$ are vertical for $\Pi$, and $\alpha$ is semi-basic.

Recall the Fundamental Lemma of Lie Groups \cite[Theorem 1.6.10]{Ivey2003}:
If there exists $\eta:\mathbf{T}\mathcal{B}\to \mathfrak{csp}(3)$ such that
$\mathrm{d}\eta+\eta\wedge\eta =0$, then for any $b \in \mathcal{B}$, there exists a
neighborhood $\mathcal{B}'$ of $b$ and a map $h:\mathcal{B}' \to CSp(3)$ such that $h^*(\mu)=\eta$.
Moreover, if $h^*(\alpha)$ is semi-basic, then the fibers of $\mathcal{B}'$ immerse into
the fibers of $CSp(3)$. 

Therefore, it suffices to construct a $\mathfrak{csp}(3)$-valued
Maurer--Cartan form $\eta$ on $\mathcal{B}$ such that the entries of the
lower-left symmetric submatrix of $\eta$, namely $h^*(\alpha_{ij})$, are semi-basic on
$\mathcal{B}$.  For brevity, the pull-backs $h^*(\cdot)$ are dropped from the
notation henceforth.

The condition that $\alpha$ is semi-basic is that $\alpha_{ij}=A_{ija}\omega^a$. 
Note that $\alpha: \mathcal{V}_4 \to \mathrm{Sym}^2(\mathbb{R}^3) =
\mathrm{Sym}^2(\mathcal{V}_2)$, and recall that for $u, v \in \mathcal{V}_2$ the symmetric tensor
is given by 
\begin{equation} 
\begin{pmatrix}
u^{-2} \\ u^{0} \\ v^{2} 
\end{pmatrix}
\circ
\begin{pmatrix}
v^{-2} \\ v^{0} \\ v^{2} 
\end{pmatrix}
=
\frac12
\begin{pmatrix}
u^{-2}v^{-2} + v^{-2}u^{-2} &
u^{-2}v^{ 0} + v^{-2}u^{ 0} &
u^{-2}v^{ 2} + v^{-2}u^{ 2} \\
u^{ 0}v^{-2} + v^{ 0}u^{-2} &
u^{ 0}v^{ 0} + v^{ 0}u^{ 0} &
u^{ 0}v^{ 2} + v^{ 0}u^{ 2} \\
u^{ 2}v^{-2} + v^{ 2}u^{-2} &
u^{ 2}v^{ 0} + v^{ 2}u^{ 0} &
u^{ 2}v^{ 2} + v^{ 2}u^{ 2} 
\end{pmatrix}.
\end{equation}
Therefore, to respect the weights in the $GL(2)$ representation, 
$\alpha$ must have the following form for
constants $A_{-4}$, $A_{-2}$, $A_{0}$, $A'_{0}$, $A_{2}$, and $A_{4}$:
\begin{equation}
\alpha = \begin{pmatrix}
A_{-4}\ \omega^{-4} & A_{-2}\ \omega^{-2} & A_{ 0}\ \omega^{ 0}\\
A_{-2}\ \omega^{-2} &A'_{ 0}\ \omega^{ 0} & A_{ 2}\ \omega^{ 2}\\
A_{ 0}\ \omega^{ 0} & A_{ 0}\ \omega^{ 2} & A_{ 4}\ \omega^{ 4}
\end{pmatrix}.
\end{equation}

Writing $\beta = \beta_\varphi(\varphi) + \beta_\lambda(\lambda) +
\beta_T(\omega)$, it is apparent that $\beta_\varphi$ must be the
representation $\mathfrak{sl}(2) \to M_{3 \times 3}(\mathbb{R})$ such that the
natural action of $M_{3 \times 3}(\mathbb{R})$ on
$\mathrm{Sym}^2(\mathcal{V}_2)$ is induced by the natural action of
$\mathfrak{sl}(2)$ on $\omega\in\mathcal{V}_4$.  One can now easily verify
the following formulas:
\begin{equation}
\begin{split}
\alpha &= \begin{pmatrix}
\omega^{-4} &\omega^{-2} &\omega^{ 0}\\
\omega^{-2} &\omega^{ 0} &\omega^{ 2}\\
\omega^{ 0} &\omega^{ 2} &\omega^{ 4}
\end{pmatrix}, \\
\beta &=  \begin{pmatrix}
4 \varphi_{0} & 2 \varphi_{2} & 0 \\
-4 \varphi_{-2} & 0 & 4\varphi_{2} \\
0 & -2 \varphi_{-2} & -4\varphi_{0}
\end{pmatrix}
-\frac12 \lambda I_3 + \beta_T(\omega).
\end{split}
\label{alpha beta}
\end{equation}
Here, the $-\frac12\lambda I_3$ component of $\beta$ is simply the scaling action of
$GL(2)$ as represented by the scaling action in $CSp(3)$.  If $PGL(2)$ and
$Sp(3)$ were used instead, it would not appear.

All that remains is to find $\beta_T$ and $\gamma$ such that $\mathrm{d}\eta +
\eta\wedge\eta=0$.  This is arithmetic, and solutions exist.  The
simplest $\eta$ (the one where the undetermined coefficients in $\beta_T$
are set to 0) is provided in \cite{GL2maplecode}.  
\end{proof}

Theorem~\ref{embed} allows explicit construction of several of the
most-symmetric Hessian hydrodynamic PDEs.  Each of the root types with three
or fewer roots is itself closed under $GL(2)$; hence, the torsion of any
$GL(2)$-structure covers the entire root type.  This allows a first step at
constructing all Hessian hydrodynamic PDEs.  However,
Theorem~\ref{embed} does not say that there are only 55 Hessian hydrodynamic
PDEs up to $CSp(3)$ actions.  There is no reason to believe that the torsion
map $T:\mathcal{B} \to \mathcal{V}_8$ is surjective for the larger root types.
When $T$ fails to be surjective, there cannot be a single PDE that represents
all possible torsions in its root type.  Nonetheless, the leaves of
dimension at most $4$ can be used to fully describe the most symmetric Hessian
hydrodynamic PDEs, as in the following theorems.

\begin{thm}
The root type $\{0\}$ is represented by the wave equation, 
\begin{equation}u_{22}=u_{13}\end{equation} and
this representation is unique up to $CSp(3)$.  Therefore,
the wave equation is the unique Hessian hydrodynamic PDE 
with nine symmetries.  
\label{PDE0} 
\end{thm}

\begin{thm}
The root type $\{8\}$ is represented by the first flow of the dKP hierarchy,
\begin{equation} u_{22}=u_{13}-\frac12\ (u_{33})^2\end{equation} and this
representation is unique up to $CSp(3)$.  Therefore, the
first flow of the dKP hierarchy is the unique hydrodynamic PDE 
with seven symmetries.  \label{PDE8} 
\end{thm}

\begin{thm}
The root type $\{7,1\}$ is represented by the PDE  \begin{equation} u_{22} =
u_{13}  - \frac1{48} u_{33}  + \frac12 u_{33}u_{23} \end{equation} and this
representation is unique up to $CSp(3)$.
\label{PDE71}
\end{thm}

\begin{thm}
The root type $\{6,2\}$ is represented by the PDE \begin{equation}
u_{22} = u_{13} + \frac{7\ u_{23}}{5u_{33}-14} \end{equation} and this
representation is unique up to $CSp(3)$.
\label{PDE62}
\end{thm}

\begin{thm}
The root type $\{6, 1, 1\}$ is represented by the PDE
\begin{equation}
u_{22}=u_{13}
+\frac{7 u_{23}\left(u_{23}-u_{33}\right)}{(5u_{33}-14)}
+\frac{49(-(u_{33})^2 + 14u_{33} - 28)}{12 (5u_{33}-14)}
-\frac{49}{6}\left(\frac{-(5u_{33}-14)}{14}\right)^{2/5}
\end{equation}
and this 
representation is unique up to $CSp(3)$.
\label{PDE611}
\end{thm}

Theorems \ref{PDE0} through \ref{PDE611} are proven via the same technique.
For illustration, consider the simplest non-trivial case,
Theorem~\ref{PDE8}.  

\begin{proof}[Proof of Theorem~\ref{PDE8}]
Consider a 2,3-integrable $GL(2)$-structure ${\pi:\mathcal{B}\to M}$ with
$T(b)=x^8\frac1{322560}$,
so $[T(\mathcal{B})]=\{8\}$.  The constant is chosen for the aesthetic appeal of the
resulting PDE.  The goal is to describe Theorem~\ref{embed}'s embedded
hypersurface $i(M')=\{ \Pi(h(q))\in \Lambda^o : \text{$q$ near $b$}\}$ as
the locus of a single equation $F(U)=0$.

As $CSp(3)^o$ is a matrix group, the open set $\{ g=h(q): \text{$q$ near
$b$}\}$ can be described as $\{ \exp_I(\eta_b(v)) : v \in \mathbf{T}_b
\mathcal{B}\}$.
Moreover, only $v \in \ker(\varphi,\lambda)$ need be considered, since the
fibers of ${\pi:\mathcal{B}\to M}$ immerse into the fibers of
${\Pi:CSp(3)^o\to\Lambda^o}$.
Fix an arbitrary $v \in \ker(\varphi,\lambda) \subset \mathbf{T}_b\mathcal{B}$ and write
$v$ in components using the tautological 1-form, $v \lhk \omega = (v_{-4},
v_{-2}, v_{0}, v_{2}, v_{4}) \in \mathbb{R}^5$.  
Of course, $v \lhk \omega$ does not actually provide local coordinates on $\mathcal{B}$
or $M$; however, the matrix $\eta_b(v)$ still represents a generic
point in $h_*(\mathbf{T}_b(\mathcal{B}))$, as seen here:
\begin{equation}
\eta_b(v)=
\begin{pmatrix}
0 & 0 & 0 & 0 & 0 & 0\\
0 & 0 & 0 & 0 & 0 & 0\\
-v_{4} & 0 & 0 & 0 & 0 & 0\\
v_{-4} & v_{-2} & v_{0} & 0 & 0 & v_{4}\\
v_{-2}& v_{0} & v_{2} & 0 & 0 & 0\\
v_{0} & v_{2} & v_{4} & 0 & 0 & 0
\end{pmatrix}.
\end{equation}
Therefore,  
\begin{equation}
\exp_I(\eta_b(v)) =
\begin{pmatrix}
1& 0& 0& 0& 0& 0\\
0& 1& 0& 0& 0& 0\\
-v_{4}& 0& 1& 0& 0& 0\\
v_{-4}-\frac{1}{6}v_{4}^3& v_{-2}+\frac{1}{2}v_{4}v_{2}& v_{0}+\frac{1}{2}v_{4}^2& 1& 0& v_{4}\\
v_{-2}-\frac{1}{2}v_{4}v_{2}& v_{0}& v_{2}& 0& 1& 0\\
v_{0}-\frac{1}{2}v_{4}^2&v_{2}& v_{4}& 0& 0& 1
\end{pmatrix}.
\end{equation}
So, using Equation~(\ref{Sp3U}), a generic point $U \in i(M')$ looks like
\begin{equation}
U=\Pi(\exp_I(\eta_p(v))) = 
\begin{pmatrix}
v_{-4}-\frac{1}{6}v_{4}^3+(v_{0}+\frac{1}{2}v_{4}^2)v_{4}& v_{-2}+\frac{1}{2}v_{4}v_{2}& v_{0}+\frac{1}{2}v_{4}^2\\
v_{-2}+\frac{1}{2}v_{4}v_{2}& v_{0}& v_{2}\\
v_{0}+\frac{1}{2}v_{4}^2& v_{2}& v_{4}
\end{pmatrix}.
\end{equation}
There is a single relation between the entries of such $U$:
\begin{equation}
U_{22} = U_{13}-\frac12\ (U_{33})^2.
\label{dKP}
\end{equation}
When $U$ is interpreted as the Hessian of $u:\mathbb{R}^3 \to \mathbb{R}$,
Equation~(\ref{dKP}) is the first flow of the dKP hierarchy, a well-known
example of a Hessian hydrodynamic equation.  
\end{proof}

The only change for the other root types is that $\eta$ is more complicated;
hence, its exponential is (immensely) more difficult to compute, and the
relation $F(U)=0$ is more difficult to recognize.  Note also that $U=0$ is
always in the locus of the equation obtained by this procedure. Therefore,
PDEs such as the Boyer--Finley equation, $u_{xx} + u_{yy} = e^{u_{tt}}$
(which has six symmetries and must represent one of the 3-dimensional
root types \cite{Ferapontov2009}), will not directly appear as representatives
via this procedure.  None-the-less, every $CSp(3)$ equivalence class of
Hessian hydrodynamic PDEs must arise this way. 

\subsection{Hyperbolic Planar PDEs}
\label{section-planar}
This section presents some preliminary but intriguing observations regarding
the hyperbolic linear Pfaffian system $\mathcal{I}$ describing bi-secant
surfaces in Theorem~\ref{bisecant} and its relation to hyperbolic second-order
planar PDEs, 
\begin{equation}
f(\xi^1, \xi^2, z, z_1, z_2, z_{11},z_{12},z_{22})=0.
\label{hypPDE}
\end{equation}
Equation~(\ref{hypPDE}) defines a 7-dimensional manifold $\Sigma_f = f^{-1}(0)
\subset \mathbb{J}^2(\mathbb{R}^2,\mathbb{R})$ whose structure equations are
obtained by pulling back the contact system \cite{Gardner1993, The2008}. Such
$\Sigma_f$ admit a point-wise classification into Monge--Amp\`ere,
Goursat, or generic-type equations.

\begin{thm}
Consider $(\mathcal{B},M,p)_{2,3}$ with $T(b)=v$
for some $b \in \mathcal{B}_p$.  Over a neighborhood $M'$ of $p$, there is a
bundle $W \to M'$ with 7-dimensional fiber and a submersion
$f:\mathcal{B}(M') \to W$ such that the ideal $\mathcal{I}$ from
Theorem~\ref{bisecant} describing the existence of bi-secant surfaces
through $p$ is the pull-back of a hyperbolic linear Pfaffian ideal
$\bar{\mathcal{I}}$ on $W$.    
This $W$ admits a coframe $(\beta^1,
\ldots, \beta^7)$ such that 
\begin{equation}
\begin{split}
\mathrm{d}\beta^1 &\equiv
\beta^2 \wedge \beta^4 + \beta^3\wedge \beta^6,\ \mod \beta^1 \\
\mathrm{d}\beta^2 &\equiv U_1\ \beta^3 \wedge\beta^7 
+ \beta^4\wedge\beta^5,\ \mod \beta^1,\beta^2\\
\mathrm{d}\beta^3 &\equiv U_2\ \beta^2\wedge\beta^5 +
\beta^6\wedge\beta^7,\ \mod \beta^1,\beta^3.
\end{split}
\label{U1U2}
\end{equation}
for $U_1 = -645120 T_{-8} \frac{(\mu_3)^3}{\mu_2\mu_1}$ and 
$U_2 = -645120 T_{ 8} \frac{(\mu_2)^3}{\mu_3\mu_1}$ for some non-zero
functions $\mu_1$, $\mu_2$, and $\mu_3$ on $W$. 

Moreover, there is some second-order hyperbolic planar PDE $f$ and a
diffeomorphism $\varphi:\Sigma_f \to W$ such that these
structure equations on $W$ pull back via $\varphi^*$ to the contact-induced
structure equations on $\Sigma_f$.
\label{W}
\end{thm}

\begin{proof}
Fix $(\mathcal{B},M,p)_{2,3}$.  As in the
general case of Theorem~\ref{bisecant}, the linear Pfaffian system describing
the existence of bi-secant surfaces through $p$ is differentially generated by
$\omega^{-2}$, $\omega^{0}$, and $\omega^{2}$, and it has tableau given by 
\begin{equation}
\mathrm{d}
\begin{pmatrix}
\omega^{-2}\\ \omega^{0}\\ \omega^{2}
\end{pmatrix}=
\begin{pmatrix}
\pi_1 & 0 \\ 0 & 0 \\ 0 & \pi_2 
\end{pmatrix}\wedge
\begin{pmatrix}
\omega^{-2}\\ \omega^{4}
\end{pmatrix}
\end{equation}
for $\pi_1= 2 \varphi_2 - 322560 T_{-2} \omega^4$ and $\pi_2= -2 \varphi_{-2}
- 322560 T_{2} \omega^{-4}$.  The Lie algebra $A(\mathcal{I})$ of Cauchy
characteristics of $\mathcal{I}$ is spanned by the duals of $\lambda$ and
$\varphi^{0}$.  Therefore, the retracting space
$C(\mathcal{I})=A(\mathcal{I})^\perp$ is a rank-seven Frobenius system on
$\mathcal{B}$, so $\mathcal{B}$ admits a foliation by 2-dimensional Cauchy
characteristic surfaces \cite[Section~II.2]{Bryant1991}
\cite[Section~6.1]{Ivey2003}.  Let $W$ denote the 7-dimensional (local) leaf
space for this foliation, so there is a submersion $\tilde{\pi}:\mathcal{B} \to W$.
Since the $\omega$ is semi-basic for the submersion $\tilde{\pi}$, $W$ also admits a
submersion onto a neighborhood of $p \in M$.
It remains to find the structure equations for a coframing on $W$.

Write $\alpha^1=\omega^{0}$, $\alpha^2=\omega^{-2}$, $\alpha^3=\omega^{2}$,
$\alpha^4 = \pi_1$, $\alpha^5=\omega^{-4}$, $\alpha^{6}=\pi_2$,
$\alpha^7=\omega^{4}$, $\alpha^8=\varphi_0$, and $\alpha^9=\lambda$ as a
coframe for $\mathcal{B}$.  To simplify the notation, fix the index
convention $1 \leq i,j,k \leq 7$ and $8 \leq r,s \leq 9$, so the forms
$\alpha^i$ are semi-basic for the bundle $\mathcal{B}\to W$, and the forms
$\alpha^r$ are vertical for the bundle $\mathcal{B}\to W$. Write
$\mathrm{d}\alpha^i = -\frac12C^i_{jk}\alpha^j\wedge\alpha^k -
C^i_{jr}\alpha^j\wedge\alpha^r$, where $C^i_{jk}=-C^i_{kj}$ and $C^i_{ja}$ are
functions of $T$ as determined by Theorem~\ref{int3-n4-str}.  Note that
$C^i_{ja}=0$ if $i\neq j$, and 
\begin{equation}
C^1_{18}=0, C^2_{28}=-4, C^3_{38}=4, C^4_{48}=4, C^5_{58}=-8, C^6_{68}=-4,
C^7_{78}=8,
\end{equation}
\begin{equation}
C^1_{19}=1, C^2_{29}=1, C^3_{39}=1, C^4_{48}=0, C^5_{58}=1, C^6_{68}=0,
C^7_{78}=1.
\end{equation}

Corollary~2.3 of \cite{Bryant1991} implies
that there exist functions $\mu_i$ on $\mathcal{B}$ such that
$\frac{1}{\mu_i}\alpha^i$ is basic.  Define a new co-framing
$(\tilde\alpha^i)$ for $\mathcal{B}$ by setting $\tilde\alpha^i =
\frac1{\mu_i}\alpha^i$ (no sum) and $\tilde{\alpha}^r=\alpha^r$.  Thus,
$\mathbf{T}^*\mathcal{B}$ is (locally) split into basic 1-forms and vertical
1-forms with respect to the bundle $\tilde{\pi}:\mathcal{B}\to W$.  Of course,
$\tilde{\alpha}^i = \tilde{\pi}^*(\beta^i)$ for some independent 1-forms
$\beta^i$, and $(\beta^1, \ldots, \beta^7)$ is the desired co-framing for $W$.
Since $\tilde{\alpha}^i$ is basic, the structure equations of
$(\tilde{\alpha}^i)$ and $(\beta^i)$ are identical. 

There is a lot of freedom in the designation of $\mu_i$.  Write
$\mathrm{d}\mu_i = \mu_{i,j}\tilde{\alpha}^j + \mu_{i,r}\tilde{\alpha}^r$.
The condition that $\tilde{\alpha}^i$ is basic is equivalent to $\mu_{i,r} =
\mu_iC^i_{ir}$ for all $1 \leq i \leq 7$ and $8 \leq r \leq 9$, but
$\mu_{i,j}$ is otherwise free, so any non-zero function $f_i$ on $W$ may be
lifted to $\mathcal{B}$ in local coordinates $(x^1, \ldots, x^9)$ by setting
$\mu_i=f_i\exp(C^i_{i8}x^8+C^i_{i9}x^9)$.

At this point, the coframe $(\beta^i)$ satisfies 
\begin{equation}
\begin{split}
\mathrm{d}\beta^1 &\equiv 0,\\
\mathrm{d}\beta^2 &\equiv \frac{\mu_4\mu_5}{\mu_2}\beta^4\wedge\beta^5,\ \mod \{\beta^1,\beta^2,\beta^3\}\\
\mathrm{d}\beta^3 &\equiv \frac{\mu_6\mu_7}{\mu_3}\beta^6\wedge\beta^7,
\end{split}
\end{equation}
so one may effectively eliminate $\mu_5$ and $\mu_7$ by redefining 
$\beta^5$ as $\frac{\mu_4\mu_5}{\mu_2}\beta^5$ and $\beta^7$ as
$\frac{\mu_6\mu_7}{\mu_3}\beta^7$.  Thus, the structure equations for $W$ satisfy
\begin{equation}
\begin{split}
\mathrm{d}\beta^1 &\equiv 0,\\
\mathrm{d}\beta^2 &\equiv \beta^4\wedge\beta^5,\ \mod \{\beta^1,\beta^2,\beta^3\}\\
\mathrm{d}\beta^3 &\equiv \beta^6\wedge\beta^7.
\end{split}
\end{equation}
One may now re-label the coframe $(\beta^i)$ following the procedure given in
Appendix A of \cite{The2008} to obtain Equation~(\ref{U1U2}).

Although $U_1$ and $U_2$ as written in the theorem are not explicitly
functions on $W$, this is easily remedied.  As noted above, $\mu_i$ may be
taken as the lift of $f_i$ on $W$.  
Also, the action of $\alpha^8$ and $\alpha^9$ is diagonal on $T$:
\begin{equation}
\mathrm{d}
\begin{pmatrix}
T_{-8}\\T_{-6}\\T_{-4}\\T_{-2}\\T_{ 0}\\T_{ 2}\\T_{4}\\T_{6}\\T_{8}
\end{pmatrix}
\equiv
\begin{pmatrix}
T_{-8}\\T_{-6}\\T_{-4}\\T_{-2}\\T_{ 0}\\T_{ 2}\\T_{4}\\T_{6}\\T_{8}
\end{pmatrix}\alpha^8 
+
\begin{pmatrix}
16T_{-8}\\12T_{-6}\\8T_{-4}\\4T_{-2}\\0\\-4T_{ 2}\\-8T_{4}\\-12T_{6}\\-16T_{8}
\end{pmatrix}\alpha^9,
\mod \alpha^1, \ldots, \alpha^7.
\end{equation}
So, while there is no natural map from $W$ to $\mathcal{O}_J(\mathcal{B})$,
the local Lie group $G$ generated by the fiber actions of $H$ and $I_9$
(corresponding to $\alpha^8$ and $\alpha^9$) induces a map $W \to
\mathcal{O}_J(\mathcal{B})/G$ that one could also call $T$.

Theorem 11.1.1 of \cite{Stormark2000} implies that
any 7-dimensional manifold with structure equations of this form must be 
diffeomorphic to $\Sigma_f$ for some hyperbolic planar PDE $f$. 
\end{proof}

\begin{cor}
For any Hessian hydrodynamic PDE, the level sets of the Riemann invariants 
are the solutions of planar hyperbolic PDEs.
\label{cor-bisecant-hyp}
\end{cor}
\begin{proof}
The manifold $W$ arises from the ideal describing bi-secant surfaces in $M$,
and Corollary~\ref{holonomic-net} shows that these are the level sets of the
Riemann invariants of the Hessian hydrodynamic PDE defining $\mathcal{B}$.
\end{proof}

This corollary is not at all surprising, since the entire point of
hydrodynamic reduction is to reduce a hyperbolic PDE in three
variables to a family of hyperbolic planar PDEs defined by the Riemann
invariants.

In the case of 2,3-integrability, Corollary~\ref{cor-bisecant-hyp} also
provides a more explicit justification for the claim in Theorem~\ref{bisecant}
that the PDE defining bi-secant surfaces can be solved in the smooth category
with smooth initial data.  One naturally asks ``to which $\Sigma_f$ is $W$
equivalent?''

\begin{cor}
Let $\Sigma_f  = f^{-1}(0) \subset \mathbb{J}^2(\mathbb{R}^2,\mathbb{R})$ for
a hyperbolic planar PDE $f$.  Suppose there is a local
diffeomorphism $\psi:\Sigma_f \to W$ with $\tilde{\pi}(b) \in \psi(\Sigma_f)$.
\begin{enumerate}
\item If $x=0$ and $y=0$ are both roots of $T(b)$, then 
$\Sigma_f$ is of the Monge--Amp\`ere type at $\psi^{-1}(\tilde{\pi}(b))$,
\item If exactly one of $x=0$ or $y=0$ is a root of $T(b)$, then
$\Sigma_f$ is of the Goursat type at $\psi^{-1}(\tilde{\pi}(b))$,
\item If neither $x=0$ nor $y=0$ is a root of $T(b)$, then
$\Sigma_f$ is of the generic type at $\psi^{-1}(\tilde{\pi}(b))$.
\label{hyperbolicPDEs}
\end{enumerate}
\label{corU1U2}
\end{cor}
\begin{proof}
These are the three pointwise types of planar hyperbolic PDEs, and
for structure equations of the form in Equation~(\ref{U1U2}), they are
determined by whether $U_1$ and $U_2$ vanish \cite{Gardner1993,The2008}.  
Since $U_1 \sim T_{-8}$, $U_1 = 0$ if and only if $y=0$ is a
root of $T(b)\in \mathcal{V}_8$.  Similarly, $U_2 =0$ if and only if $x=0$ is
a root of $T(b) \in \mathcal{V}_8$.
\end{proof}

Corollary~\ref{corU1U2} puts interesting restrictions on which
$GL(2)$-structures can yield which planar PDEs.  In particular, $x=0$ and $y=0$
can both be roots of $v=T(b)$ if and only if $[v]$ is a root type having two
distinct real roots, and the strictly complex root types cannot have $x=0$ or
$y=0$ as roots.  Notably, the type of $W$ can change.  For example, suppose
$T(b)=x^7y$, so $\psi^{-1}(\tilde{\pi}(b))$ is of the Monge--Amp\`ere type but
nearby $\psi^{-1}(\tilde{\pi}(b'))$ is of the generic type, since
$T(b')=(x+\varepsilon_1y)^7(\varepsilon_2x + y)$.  This type-changing does not
occur for the flat structure, and it is easy to compute a change-of-frame from
the structure equations of the flat $W$ to the structure equations of
$\Sigma_f$ for the planar wave equation $z_{12}=0$.

\begin{cor}
Let $(\mathcal{B},M,p)_{2,3}$ be flat. Then $W$ is
isomorphic to $\{z_{12}=0\} \subset
\mathbb{J}^2(\mathbb{R}^2,\mathbb{R})$.
\end{cor}

\section{Concluding Remarks}
The main results of this article are summarized in
Figure~\ref{fig:classes-equiv}.  In short, Hessian hydrodynamic PDEs in
three independent variables are equivalent to local 2,3-integrable
$GL(2)$-structures of degree 4, and both objects admit a geometric,
coordinate-free classification by the singular foliation of $\mathbb{R}^9$ shown in the
figure.

\begin{figure}
\begin{centering}
\includegraphics[width=\textwidth]{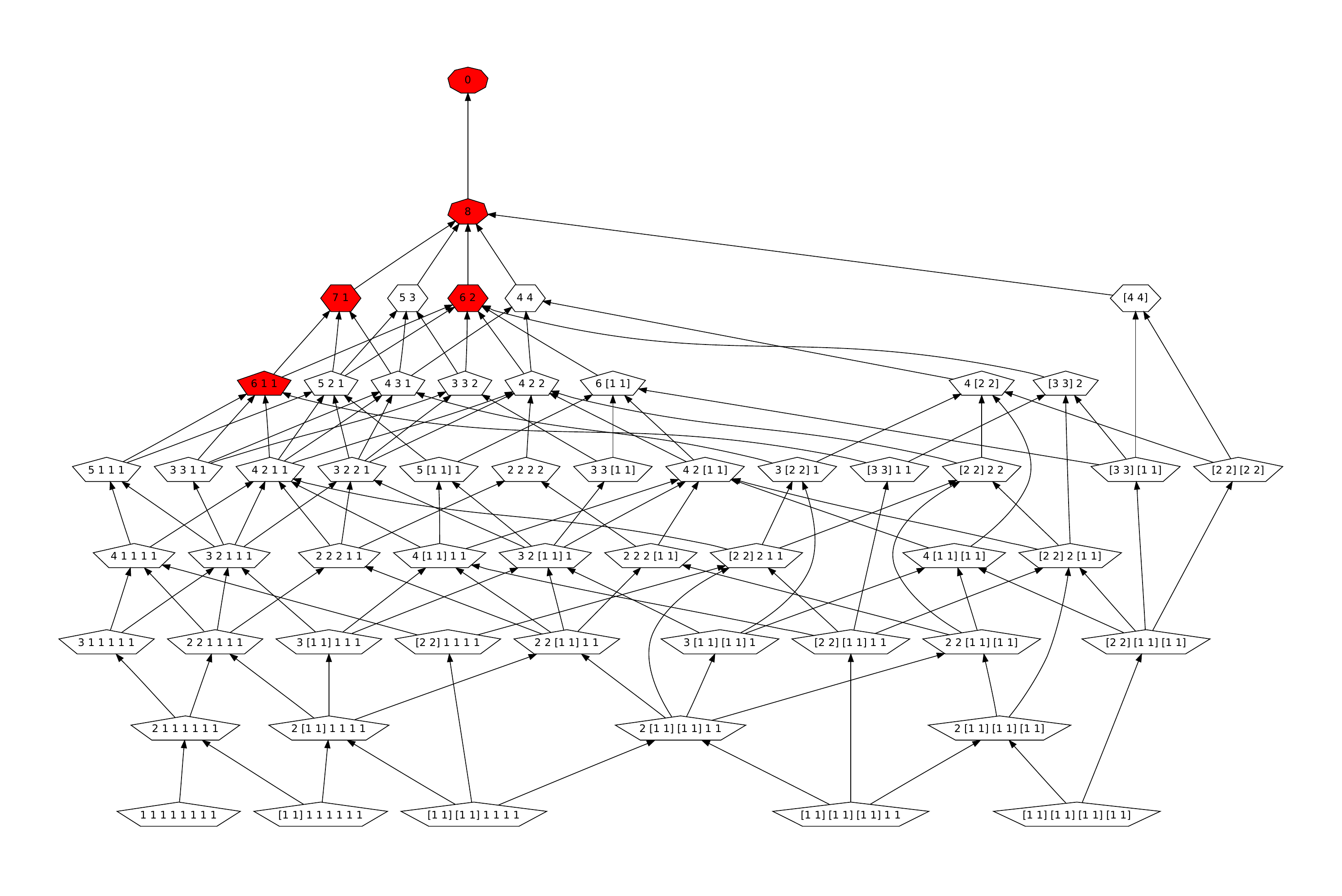}
\caption{The leaf-classification of all $(\mathcal{B},M,p)_{2,3}$.
The number of sides on a node is the dimension of the bundle after the symmetry
reduction from Section~\ref{symmred}. 
Representative PDEs for the shaded nodes are included in Section~\ref{PDEs}.}
\label{fig:classes-equiv} 
\end{centering}
\end{figure}

Lemma~\ref{matrixmiracles} seems to be a miraculous coincidence.  The
bluntness of this relationship between $v$ and $J(v)$ prompted me to
investigate relationships between the roots of $v$ and the structure of the
leaf $\mathcal{O}_J(v)$, yielding this project's main result,
Theorem~\ref{classification}.  It appears that no such relationship holds for
2,3-integrable $GL(2)$-structures of degree $n \geq 5$, even though versions of
Theorem~\ref{int3-n4-str} and Lemma~\ref{leaves} exist in those cases.  In
general, intransitive groupoids and pseudo-groups are very poorly understood,
and it is generally impossible to explicitly write down the integral manifolds
of any given singular distribution.  If it were not for the coincidence in
this case, the leaf-equivalence classes would have little utility in
understanding the Hessian hydrodynamic PDEs.

Of course, Hessian hydrodynamic PDEs on $u:\mathbb{R}^3 \to \mathbb{R}$ are
not the only integrable PDEs of interest in mathematics and physics.   Two
generalizations are important to consider:
\begin{enumerate}
\item[Q1] Integrability should be a contact-invariant property of a PDE.  The
Hessian-only form
$F(u_{11},u_{12}, u_{13}, u_{22}, u_{23}, u_{33})=0$ is not preserved under
the full family of contact transformations.  However, the associated
$GL(2)$-structures apparently relied on this form and its associated $CSp(3)$
transformations.  How can this classification be extended to second-order PDEs
in three independent variables that also include lower derivatives?  How does
the $GL(2)$ geometry generalize to these PDEs?
The observations and computations of some recent articles may prove 
very useful \cite{Burovskiy2008} \cite{Ferapontov2003}.

\item[Q2] Very few integrable PDEs are known to exist in more than three
independent variables, but equations of the form $F(u_{11}, \ldots, u_{NN})=0$
can sometimes yield distributions of rational normal cones of degree $n$ on
hypersurfaces $M^n =F^{-1}(0) \subset \mathrm{Sym}^2(\mathbb{R}^N)$ with
$n=\frac12N(N+1)-2$.  Results similar to those in Section~\ref{3int} are known
for 2,3-integrability for degrees 5 through 20, but the foliation by groupoid
    orbits of $\mathcal{V}_{n+4}$ is not understood, and $k$-integrability is
    extremely restrictive \cite{Smith2009}.  What can integrable GL(2)
    geometry say about the existence of integrable PDEs in more variables?
    \end{enumerate}

On a more detailed level, it would be interesting to study the foliation that
appears in the present case.  Despite the significant computational
difficulties, it is important both to produce representative PDEs for each
root type and to find the root types of the well-known Hessian hydrodynamic
PDEs.  For example, if one can produce Hessian hydrodynamic PDEs as well as
contact-equivalent PDEs that involve lower-order terms, such computations
could open the door to Q1, above.  Additionally, the lack of surjectivity of
$T:\mathcal{B}\to \mathcal{O}_J(\mathcal{B})$ is irritating.  What is the
exact relationship between two PDEs that are leaf-equivalent but do not have
overlapping torsion? 

Finally, the relationship between $W$, which describes bi-secant surfaces in
$M$, and $\Sigma_f$, which arises from a planar hyperbolic PDE, is worth
pursuing.  What is is the nature of the correspondence?  Can every hyperbolic
planar PDE appear this way? Can this correspondence provide any new
information about the Riemann invariants of the Hessian hydrodynamic PDE?

\appendix

\section{The Matrix J(T)}\label{appxJ}
For reference, here is the matrix $J(T)$, listed by column.  Note! In the $\omega$
columns, the common factor of $9216=2^{10}3^2$ has been removed for clarity.

The $9216\ \omega^{-4}$ column:
\[\begin{pmatrix}280T_{-8}T_{4}-280T_{-6}T_{2}\\
 -245T_{-4}T_{2}+70T_{-8}T_{6}+175T_{-6}T_{4}\\
 70T_{-4}T_{4}-210T_{-2}T_{2}+130T_{-6}T_{6}+10T_{-8}T_{8}\\
 -175T_{0}T_{2}-35T_{-2}T_{4}+35T_{-6}T_{8}+175T_{-4}T_{6}\\
 84T_{-4}T_{8}+196T_{-2}T_{6}-140T_{2}^2-140T_{0}T_{4}\\
 -350T_{2}T_{4}+175T_{0}T_{6}+175T_{-2}T_{8}\\
 350T_{0}T_{8}-350T_{4}^2\\
 700T_{2}T_{8}-700T_{6}T_{4}\\
 -1400T_{6}^2+1400T_{8}T_{4}\end{pmatrix}\]

The $9216\ \omega^{-2}$ column:
\[\begin{pmatrix}-1400T_{-8}T_{2}+1400T_{-6}T_{0}\\
 -385T_{-8}T_{4}-840T_{-6}T_{2}+1225T_{-4}T_{0}\\
 -280T_{-4}T_{2}+1050T_{-2}T_{0}-70T_{-8}T_{6}-700T_{-6}T_{4}\\
 -910T_{-4}T_{4}+280T_{-2}T_{2}+875T_{0}^2-240T_{-6}T_{6}-5T_{-8}T_{8}\\
 1540T_{0}T_{2}-952T_{-2}T_{4}-28T_{-6}T_{8}-560T_{-4}T_{6}\\
 -105T_{-4}T_{8}-1120T_{-2}T_{6}+1400T_{2}^2-175T_{0}T_{4}\\
 2100T_{2}T_{4}-1750T_{0}T_{6}-350T_{-2}T_{8}\\
 2450T_{4}^2-1050T_{0}T_{8}-1400T_{2}T_{6}\\
 -2800T_{2}T_{8}+2800T_{6}T_{4}\end{pmatrix}\]

The $9216\ \omega^{ 0}$ column:
\[\begin{pmatrix}-2800T_{-6}T_{-2}+2800T_{-8}T_{0}\\
 -2450T_{-4}T_{-2}+875T_{-8}T_{2}+1575T_{-6}T_{0}\\
 210T_{-8}T_{4}-2100T_{-2}^2+1540T_{-6}T_{2}+350T_{-4}T_{0}\\
 1890T_{-4}T_{2}-2625T_{-2}T_{0}+35T_{-8}T_{6}+700T_{-6}T_{4}\\
 1568T_{-4}T_{4}+336T_{-2}T_{2}-2100T_{0}^2+192T_{-6}T_{6}+4T_{-8}T_{8}\\
 -2625T_{0}T_{2}+1890T_{-2}T_{4}+35T_{-6}T_{8}+700T_{-4}T_{6}\\
 210T_{-4}T_{8}+1540T_{-2}T_{6}-2100T_{2}^2+350T_{0}T_{4}\\
 -2450T_{2}T_{4}+1575T_{0}T_{6}+875T_{-2}T_{8}\\
 2800T_{0}T_{8}-2800T_{2}T_{6}\end{pmatrix}\]

The $9216\ \omega^{ 2}$ column:
\[\begin{pmatrix}2800T_{-4}T_{-6}-2800T_{-2}T_{-8}\\
 -1400T_{-6}T_{-2}-1050T_{-8}T_{0}+2450T_{-4}^2\\
 2100T_{-4}T_{-2}-350T_{-8}T_{2}-1750T_{-6}T_{0}\\
 -105T_{-8}T_{4}+1400T_{-2}^2-1120T_{-6}T_{2}-175T_{-4}T_{0}\\
 -952T_{-4}T_{2}+1540T_{-2}T_{0}-28T_{-8}T_{6}-560T_{-6}T_{4}\\
 -910T_{-4}T_{4}+280T_{-2}T_{2}+875T_{0}^2-240T_{-6}T_{6}-5T_{-8}T_{8}\\
 1050T_{0}T_{2}-280T_{-2}T_{4}-70T_{-6}T_{8}-700T_{-4}T_{6}\\
 1225T_{0}T_{4}-385T_{-4}T_{8}-840T_{-2}T_{6}\\
 1400T_{0}T_{6}-1400T_{-2}T_{8}\end{pmatrix}\]

The $9216\ \omega^{ 4}$ column:
\[\begin{pmatrix}1400T_{-4}T_{-8}-1400T_{-6}^2\\
 -700T_{-4}T_{-6}+700T_{-2}T_{-8}\\
 -350T_{-4}^2+350T_{-8}T_{0}\\
 -350T_{-4}T_{-2}+175T_{-8}T_{2}+175T_{-6}T_{0}\\
 84T_{-8}T_{4}-140T_{-2}^2+196T_{-6}T_{2}-140T_{-4}T_{0}\\
 -35T_{-4}T_{2}-175T_{-2}T_{0}+35T_{-8}T_{6}+175T_{-6}T_{4}\\
 70T_{-4}T_{4}-210T_{-2}T_{2}+130T_{-6}T_{6}+10T_{-8}T_{8}\\
 70T_{-6}T_{8}-245T_{-2}T_{4}+175T_{-4}T_{6}\\
 -280T_{-2}T_{6}+280T_{-4}T_{8}\end{pmatrix}\]

The $\lambda$, $\varphi_{-2}$, $\varphi_{0}$, and $\varphi_{2}$ columns:
\[\begin{pmatrix}T_{-8}\\
 T_{-6}\\
 T_{-4}\\
 T_{-2}\\
 T_{0}\\
 T_{2}\\
 T_{4}\\
 T_{6}\\
 T_{8}\end{pmatrix},
\begin{pmatrix}-16T_{-6}\\
 -14T_{-4}\\
 -12T_{-2}\\
 -10T_{0}\\
 -8T_{2}\\
 -6T_{4}\\
 -4T_{6}\\
 -2T_{8}\\
 0\end{pmatrix},
\begin{pmatrix}16T_{-8}\\
 12T_{-6}\\
 8T_{-4}\\
 4T_{-2}\\
 0\\
 -4T_{2}\\
 -8T_{4}\\
 -12T_{6}\\
 -16T_{8}\end{pmatrix},
\begin{pmatrix}0\\
 2T_{-8}\\
 4T_{-6}\\
 6T_{-4}\\
 8T_{-2}\\
 10T_{0}\\
 12T_{2}\\
 14T_{4}\\
 16T_{6}\end{pmatrix}\]

\bibliographystyle{amsplain}
\bibliography{GL2-PDEs}

\end{document}